\documentclass[a4paper,10pt]{article}
 
\usepackage[latin1]{inputenc}
\usepackage{enumerate}

\usepackage{amsmath}
\usepackage{amsfonts,amssymb}
\usepackage{amsthm}
\usepackage[all]{xy}
\usepackage{array}
\usepackage{color}
\usepackage{mathabx}
\usepackage{ulem}
\usepackage{dsfont}
\usepackage{bbm}

\newtheorem{theorem}[equation]{Theorem}
\newtheorem{lemma}[equation]{Lemma}
\newtheorem{proposition}[equation]{Proposition}
\newtheorem{corollary}[equation]{Corollary}
\newtheorem{hypothesis}[equation]{Hypothesis}

\theoremstyle{definition}
\newtheorem{definition}[equation]{Definition}

\newtheorem{remark}[equation]{Remark}

\newcommand{\abs}[1]{\left\vert#1\right\vert}

\newcommand{\X}{\ensuremath{\mathcal{X}}}

\newcommand{\E}{\ensuremath{\mathcal{E}}}
\newcommand{\D}{\ensuremath{\mathcal{D}}}
\newcommand{\C}{\ensuremath{\mathcal{C}}}
\renewcommand{\P}{\ensuremath{\mathcal{P}}}

\DeclareMathOperator{\hocmp}{hocmp}
\DeclareMathOperator{\cmp}{cmp}
\DeclareMathOperator{\hofiber}{hofiber}
\DeclareMathOperator{\fiber}{fiber}
\DeclareMathOperator{\hocofiber}{hocof}
\DeclareMathOperator{\hocof}{hocof}

\DeclareMathOperator{\holim}{holim}
\DeclareMathOperator{\hocolim}{hocolim}
\DeclareMathOperator{\colim}{colim}

\numberwithin{equation}{section} 

\author{Lu\'is Alexandre Pereira}

\begin{document}

\title{A general context for Goodwillie Calculus}

\maketitle

The objective of this paper is to provide a fairly general model category context in which one can perform Goodwille Calculus. 

There are two main parts to the paper, the first establishing general conditions which guarantee the existence of universal $n$-excisive approximations, and a second providing conditions for which the layers of the Goodwillie tower can be appropriately factored through the spectra categories.


\section{Overview}

The main goal of this paper is to establish a fairly general setup in which to do Goodwillie calculus, extending the treatment of Goodwillie in \cite{Goodwillie}.

There are two main results. The first is the existence of universal $n$-excisive approximations, which is Theorem \ref{univ}. The second is the classification of $n$-homogeneous functors $F \colon \C \to \D$ (at least under some finitary conditions) as having the form $F = \Omega^\infty((\bar{F}\circ (\Sigma^\infty)^{\times n})_{h \Sigma_n})$, where $\bar{F} \colon Stab(\C)^{\times n} \to Stab(\D)$ is a symmetric multilinear functor between the stabilizations. In the current write-up, this result corresponds to piecing together Proposition \ref{factspectratar}, which shows such functors factor through $Stab(\D)$ (provided they are determined by their values on a small subcategory $\C' \subset \C$), Proposition \ref{homogarecross}, showing that, provided the target category is stable, a homogeneous functor is recovered from it's cross effect, and Theorem \ref{multfact}, which proves directly that multilinear symmetric functors factor through stabilizations on both source and target.

The following is a more detailed overview of the paper.

Section \ref{setup} establishes some of the general assumptions on model categories and functors between them used throughout the paper. Particularly important is Proposition \ref{comholim}, stating informally that, in any cofibrant model category, ``linear/direct'' colimits over a large enough ordinal commute with finite holims. 

Section \ref{polyfunct} establishes the basic definitions of excisive functors (via hocartesian/hococartesian cubes), and many of their basic properties, which are generally straightforward generalizations of the analogous properties found in \cite{Goodwillie}. Of some technical interest seems to be Proposition \ref{comcones}, which is currently necessary for the proof of Theorem \ref{univ} (check Remark \ref{extracond} for more on this). 
 
Section \ref{taylor} then establishes the first main result, Theorem \ref{univ}, for the existence of universal $n$-excisive approximations $P_nF$. These are constructed via Goodwillie's $T_nF$, except those are iterated a transfinite number of times. Excisiveness of this construction then follows by Goodwillie's usual proof combined with Proposition \ref{comholim} and holds generally under the assumption that $\C,\D$ are cofibrantly generated model categories. The proof of universality seems currently more delicate, however, relying on Proposition \ref{comcones}, which currently requires the source category $\C$ to be either a pointed simplicial model category or a left Bousfield localization of one.

Section \ref{goodstab} establishes that, when the target category $\D$ is stable, homogeneous functors and multilinear functors determine each other (Propositions \ref{multarecross} and \ref{homogarecross}), and is a fairly straightforward generalization of the treatment in \cite{Goodwillie}. This requires no further assumptions on the categories.

Section \ref{delooping} establishes that homogeneous functors can be delooped. This \emph{does} require additional conditions on the categories. Namely, it is rather unclear how to generalize Goodwillie's proof of this when the $T_n$ are to be iterated a transfinite number of times, so that one is forced to assume $P_n$ can be constructed via countable iterations, and this, in turn, essentially amounts to assuming that, in the target category $\D$, countable directed hocolims commute with finite holims. It also establishes the infinitely iterated result, Proposition \ref{factspectratar}.

Finally, Section \ref{factspectra}, deals with showing Theorem \ref{multfact}. This too imposes more conditions on the model categories, most obviously requiring that the stabilizations (as defined in \cite{Hovey}) exist, and the further technical condition that in those stabilizations one can detect weak equivalences by looking at the $\Omega^{\infty -i}$ functors (see \cite{Hovey} for a discussion of this condition).

\section{Setup for homotopy calculus}\label{setup}

This section lists assumptions that appear repeatedly in the paper, and deduces some of their basic consequences.

Pervasive throughout the paper is the assumption that the model categories used are cofibrantly generated. We recall the definition (and set notation) in \ref{cofgen}, along with some results about the functoriality of hocolimits and holimits in such model categories.

\ref{assumpfunct} deals with the general assumptions made on functors between model categories.

\ref{assumpsimp} deals with the further assumptions made when the model categories are also assumed simplicial.

\ref{assumpspectra} deals with detailing the notions of spectra in a general simplicial model category that are used in Section \ref{factspectra}.

Finally, \ref{assumpquillenequiv} deals with showing that our definitions are nicely compatible with replacing one of the model categories involved by a Quillen equivalent one.

\subsection{Cofibrant generation}\label{cofgen}

\begin{definition}
A model category $\C$ is said to be cofibrantly generated if there exist sets of maps $I$ and $J$, and regular cardinal $\kappa$ such that

\begin{enumerate}
\item The domains of $I$ are small relative to $I$ with respect to $\kappa$.
\item The trivial fibrations are precisely the maps with the RLP\footnote{Right Lifting Property.} with respect to $I$.
\item The domains of $J$ are small relative to $J$ with respect to $\kappa$.
\item The trivial fibrations are precisely the maps with the RLP with respect to $J$.
\item Both the domains and targets of $I$ and $J$ are small relative to $I$ with respect to $\kappa$.
\end{enumerate}
\end{definition}

\begin{remark}
We refer to \cite{Hirsch}, section 10, for the precise definitions of ``smallness with respect to'', but we give a sketch definition here: 

$X$ is said to be $\kappa$-small with respect to $I$ if for any $\lambda \geq \kappa$ regular cardinal\footnote{Here, as in \cite{Hirsch}, cardinals are to be viewed as ordinals by selecting the initial ordinal with that cardinality.} and $\lambda$-sequence\footnote{We recall that a functor $Y_.\colon \lambda \to \C$ is called a $\lambda$-sequence if it is ``well behaved'' on limit ordinals, i.e. $Y_\alpha=\colim_{\beta < \alpha}Y_{\beta}$ for $\alpha<\gamma$ a limit ordinal.}
\[Y_0\to Y_1 \to Y_2 \to \dots \to Y_\beta \to Y_{\beta +1} \to \dots (\beta < \lambda)\]
where all maps $Y_\beta \to Y_{\beta +1}$ are relative $I$-cell complexes\footnote{The more ``simplicially minded'' reader can likely ignore this, as in those contexts objects are often small with respect to all maps.}, the natural map 
\[\colim_{\beta < \lambda} Hom(X,Y_\beta)\xrightarrow{\sim} Hom(X,\colim_{\beta < \lambda}Y_\beta)\]
is an equivalence.
\end{remark}

\begin{remark}
Though the first condition in our definition is a particular case of the last, we choose to include it so as to draw attention to the fact that that condition is not present in the definition of cofibrantly generated model category in \cite{Hirsch}.

\end{remark}

Our goal is to implement Goodwillie's construction of the polynomial approximations in this general context. Cofibrant generation allows us to do so, thanks to the following results:

\begin{proposition}\label{functrep}
Suppose $I$ a set of maps whose domains are small relative to $I$ with respect to a regular cardinal $\kappa$.

Then there exists a functorial factorization
\[\C^\downarrow \to \C^\downarrow \times \C^\downarrow\]
factoring any map into a relative $I$-cell complex followed by an $I$-injective.
\end{proposition}

\begin{proof}
This is just Quillen's well known small object argument. A reference for this is section 2.1.2 of \cite{HoveyMC}, or most other basic treatments of model categories.
\end{proof}

Note that, according to the previous result, in any general finitely generated model category $\C$, one can construct functorial cofibrant and fibrant replacement functors. We will denote fixed such functors by $Q$ and $R$ in the remainder of this paper.

We will also need to make special note of particularly nice types of diagram categories that feature prominently in Goodwillie calculus (see \cite{Clark} for the following definitions):

\begin{definition}\label{dirinv}
A small category $A$ is called a {\bf direct category} if there is an identity reflecting functor $A \to \lambda$, where $\lambda$ is an ordinal category.

A small category $A$ is called an {\bf inverse category} if $A^{op}$ is a direct category.
\end{definition}

It is well known that when $A$ is direct a projective model structure exists on any $\C^A$, for $\C$ any model category (check \cite{Clark} for the characterization of this model structure). Dually, when $A$ is an inverse category, $\C^A$ always has an injective model structure. We will use the following basic properties about such model structures:

\begin{proposition}\label{restprop}
Let $A$ be a direct category and $A'$ an initial subcategory (i.e., such that if $x \in A'$ then all maps $x' \to x$ are in $A'$). Note that $A'$ is then itself a direct category. Then if $F \colon A \to \C$ is projective cofibrant, then so is the restriction $F|_{A'}$. Furthermore, the analogous result holds for cofibrations $F \rightarrowtail F'$.

Let $A,B$ be direct categories. Then $A\times B$ is a direct category and, if $F \colon A \times B \to \C$ is projective cofibrant, then so are the restrictions $F(a,\bullet)$, $F(\bullet,b)$ for any $a \in A, b \in B$. Furthermore, the analogous result holds for cofibrations $F \rightarrowtail F'$.
\end{proposition}

\begin{proof}
The first part is straightforward from the definition of the model structure, as the conditions required for $F|_{A'} \to F'|_{A'}$ to be a cofibration are a subset of the conditions for $F \to F'$ to be one.

As for the second part, it follows from noticing that in these projective model structures cofibrations are in particular pointwise cofibrations, and by further viewing the projective model structure on $\C^{A\times B}$ as the projective model structure in $(\C^A)^B$ over the projective model structure in $\C^A$.
\end{proof}

\begin{proposition}
Let $A$ be a direct category, and $\C$ a model category with functorial factorizations as defined in Proposition \ref{functrep}. Then so does the projective model structure on $\C^A$. Moreover, such functorial factorizations can be constructed using those in $\C$.

Furthermore, if $\C$ is assumed simplicial, and with simplicial factorization functors, then $\C^A$ is itself a simplicial model category and the constructed factorizations are again simplicial functors.
\end{proposition}

\begin{proof}
This statement is a variation of Proposition 6.3 of \cite{RSS}, and the proof is essentially the same.
\end{proof}

\begin{proposition} \label{comholim}
Let $\C$ be a finitely generated model category with respect to some regular cardinal $\kappa$. Then $\kappa$-filtered colimits (taken over the $\kappa$ ordinal category) commute with $\kappa$-small limits (taken over inverse categories).

\end{proposition}

Before giving the proof, we recall some definitions. For more details, see \cite{Emily}. Given a functor 
\[\C \xrightarrow{F} \D\]
between model categories $\C$ and $\D$, we define the derived functors $\mathbb{L}F,\mathbb{R}F \colon Ho(\C) \to Ho(\D)$ by
\[\mathbb{L}F=Ran_{\gamma_{\C}}\gamma_{\D}\circ F, \quad \mathbb{R}F=Lan_{\gamma_{\C}}\gamma_{\D}\circ F\]
where $\gamma_{\C} \colon \C \to Ho(\C)$, $\gamma_{\D} \colon \D \to Ho(\D)$ are the standard maps.

It is then well known that, when $F$ is a left Quillen (resp. right Quillen) functor, $\mathbb{L}F$ (resp. $\mathbb{R}F$) is the functor induced by $F \circ Q$ (resp. $F \circ R$).

The statement of \ref{comholim} is then that the following diagram (where $A$ is assumed the ordinal category $\kappa$ and $B$ assumed inverse and $\kappa$-small) commutes up to a natural isomorphism.

\begin{equation}\xymatrix{
Ho(\C^{A\times B}) \ar[d]^{\mathbb{L}\varinjlim_A}\ar[r]^{\mathbb{R}\varprojlim_B} & Ho(\C^A) \ar[d]^{\mathbb{L}\varinjlim_A}\\
Ho(\C^B) \ar[r]^{\mathbb{R}\varprojlim_B} & Ho(\C) 
}\label{homcom}\end{equation}

Note, however, that we do not claim the natural isomorphism supplied in the proof can be chosen canonically (this should be of contrasted with the canonical map $\varinjlim_A \varprojlim_B\to \varprojlim_B \varinjlim_A$ ), as it depends at this level on the choice of bifibrant replacements used. 

\begin{proof}[Proof of \ref{comholim}]
The proof will be an adaptation of Resk's argument in \cite{Schwede}, Lemma 1.3.2.

First suppose given a diagram $X_{a,b}$ such that $X_{\bullet,b}$ is injective cofibrant for all $b$ and $X_{a,\bullet}$ is projective fibrant for all $a$. Then we claim than $\varinjlim_a X_{a,b}$ is still projective fibrant. By the assumption that $B$ be an inverse category, this amounts to showing that the maps $\varinjlim_a X_{a,b} \to \varprojlim_{ b \xrightarrow{f} b',f\neq id_b} \varinjlim_a X_{a,b}X_{a,b'}$ are fibrations. Checking this amounts to checking the RLP of these maps with respect to the generating trivial cofibrations $J$. Suppose then $j \to j'$ an element of $J$. A commutative square corresponds to maps $j \to \varinjlim_a X_{a,b}$ and (for each $b \xrightarrow{f} b'$) $j' \to \varinjlim_a X_{a,b}X_{a,b'}$. But compactness allows us to factor each of these maps $j \to X_{a,b}$, etc by some $a$ which can be assumed uniform since $A$ is $\kappa$-filtered. But then a further usage of compactness (and filteredness) tells us that, possibly rechoosing $a$, the commutative square factors through $ X_{a,b} \to \varprojlim_{ b \xrightarrow{f} b',f\neq id_b} X_{a,b'}$, for some $a$, and hence the lifting property follows.

We can now prove the result. First notice that, by first taking a functorial bifibrant replacement of the $A \times B$ diagram, one can assume the first maps in \eqref{homcom} are taken to be straight up limits and colimits. Furthermore, by the further part of the proof, applying $\varinjlim_A$ preserves projective fibrancy, so the second $\varprojlim_B$ can also the taken straight up. 

In other words, taking $\bar{X}_a$ an injective cofibrant replacement to $\varprojlim_b X_{a,b}$, we will be done provided the natural map
\[\varinjlim_a \bar{X}_a \to \varprojlim_b \varinjlim_a X_{a,b}\]
is a weak equivalence. More specifically, we claim it is a trivial fibration. Indeed, letting $i\to i'$ an element of the generating cofibrations $I$, the techniques used in the first part of the proof allow us to factor any commutative square through $\bar{X}_a \to \varprojlim_b X_{a,b}$ for some fixed $a$, and the result follows.

\end{proof}

\begin{remark}

Notice that the previous result also follows (with no alterations to the proof) for other $\kappa$-filtered shapes of the diagram category $A$ provided one knows the domains and codomains of $I$ and $J$ satisfy compactness with respect to such shapes.
\end{remark}

\subsection{Functors}\label{assumpfunct}

We now indicate our assumptions on functors. Throughout we will denote
\[F \colon \C \to \D\]
a functor between a simplicial model category $\C$ and a model category $\D$ (both still assumed cofibrantly generated), which is typically assumed left homotopical, i.e., such that $F\circ Q$ is homotopical\footnote{Recall that a functor is called homotopical if it preserves w.e.s.}. The reason for this slight generalization is that, even for straight up homotopical functors, some of our constructions require precomposition with $Q$ anyway.

We do however notice that, when we allow $F$ to be merely left homotopical, our constructions will be presenting the Goodwillie tower of $F \circ Q$, not of $F$.

\subsection{Simplicial categories and functors}\label{assumpsimp}

In the situation in which we are dealing with simplicial model categories $\C,\D$, we will sometimes (but not always) require the functor $F\colon \C \to \D$ between them to be a simplicial functor. When that is the case we would also like to know that our constructions of the Goodwillie tower still yield simplicial functors. Since cofibrant replacements are necessary when making those constructions, we need the following result, which is proposition 6.3 in \cite{RSS}:

\begin{proposition}
Suppose $\C$ a cofibrantly generated simplicial model category.
Then there exists simplicially functorial factorizations
\[\C^\downarrow \to \C^\downarrow \times \C^\downarrow\]
factoring any map into a trivial cofibration followed by fibration, or as a cofibration followed by trivial fibration.

In particular, $\C$ has simplicial cofibrant and fibrant replacement functors.
\end{proposition}

When dealing with simplicial model categories we will hence always further assume that the chosen replacement functors $Q,R$ are taken to be simplicial.

\begin{remark}
Note that the definition of ``cofibrantly generated'' given in this paper is slightly more demanding than that used in \cite{RSS}.
\end{remark}

\subsection{Spectra on model categories}\label{assumpspectra}

We now explain the notions of spectra that will be used in Section \ref{factspectra}.

Assume $\C$ a pointed simplicial model category, so that it is also in particular tensored over $SSet_*$. Denote this tensoring by $\wedge$.

Following Hovey in \cite{Hovey} one can then define Bousfield-Friedlander type spectra on $\C$ as sequences $(X_i)_{i\in \mathbb{N}}$, plus structure maps 
\[S^1 \wedge X_n \to X_{n+1},\]
where $S^1$ is the standard simplicial circle $\Delta^1/\delta\Delta^1$.

We denote the category of such spectra by $Sp(\C,S^1)$, or just $Sp(\C)$ when no confusion should arise.

Then when $\C$ is cofibrantly generated one has that $Sp(\C)$ always has a cofibrantly generated projective model structure (i.e. a model structure where weak equivalences and fibrations are defined levelwise on the underlying sequences of the spectra), with generating cofibrations and generating trivial cofibrations given respectively as \[I_{proj}=\bigcup_n F_n I,\quad  J_{proj}=\bigcup_n F_n J,\] where $I,J$ denote respectively the generating cofibrations and generating trivial cofibrations of $\C$, and $F_n$ is the level $n$ suspension spectra functor.

To get the correct notion of stable w.e.'s on spectra one then needs to left Bousfield localize with respect to the maps 
\[S= \{F_{n+1}(S^1\wedge A_i) \to F_n A_i)\}\]
where $A_i$ ranges over the (cofibrant replacements of, when necessary) domains and codomains of the maps in $I$.

\begin{definition}\label{descgen}
Let $\C$ be a cofibrantly generated pointed simplicial model category. 

Then we define the stable model structure on $Sp(\C)$ as the 
(cofibrantly generated) left Bousfield localization with respect to $S$, should it exist.
\end{definition}

\begin{remark}
For a detailed treatment of left Bousfield localizations, check \cite{Hirsch}, chapters 3 and 4. 

Notice that we crucially do not want to just assume $\C$ is left proper cellular, as in \cite{Hovey}, where it is shown that that is sufficient for the model structure above to exist.

This is because left properness generally fails in the main examples we are interested in, those of algebras over an operad in symmetric spectra. This turns out to be ok, however, both because one can show directly in those cases that the description above produces a model category, and because even though Hovey requires left properness in many of his statements in \cite{Hovey}, in the ones we shall need he is only really requiring the existence of the model structure just described.
\end{remark}

We now adapt the treatment in \cite{Hirsch} in order to show that one has a nice localization functor in $Sp(\C)$. 

First notice that the projective model structure on $Sp(\C)$ is simplicial (just apply the simplicial constructions levelwise). 

Then set 
\[\tilde{S}=\{ F_{n+1}(S^1\wedge A_i) \to \widetilde{F_n A_i}\}\]
the set of maps obtained by taking a cofibrant factor of those in $S$.

And set 
$\widetilde{\Delta S}= J_{proj} \cup \{F_{n+1}(S^1\wedge A_i)\otimes \Delta_m \coprod_{\{F_{n+1}(S^1\wedge A_i)\otimes \delta \Delta_m} \widetilde{F_n A_i}\otimes \delta\Delta_m \to  \widetilde{F_n A_i}\otimes \Delta_m \}$.

\begin{proposition}\label{localization}
Let $X\in Sp(\C)$. Then the localization of $X$ can be obtained by performing the Quillen small object argument on the map $X\to *$ with respect to the set of maps  $\widetilde{\Delta S}$.

\end{proposition}

\begin{proof}
This is just Proposition 4.2.4 from \cite{Hirsch}. It should be noted that though the statement asks for the category to be left proper cellular none of those properties is used, and indeed all one needs is cofibrant generation (in order to run the small object argument).

\end{proof}

We will generally denote the by 
\[\Sigma^{\infty} \colon \C \rightleftarrows Sp(\C) \colon \Omega^{\infty}\]
the standard adjoint functors.

We will also require an additional niceness property about the model structure on $Sp(\C)$, namely, we shall require that the stable equivalences be detectable via the $\tilde{\Omega}^{\infty - n}$ functors (these are the derived functors giving the $n$-th space of the spectrum after being made into an omega spectrum).

In \cite{Hovey} Hovey determines conditions under which this follows (and, importantly, those proofs do not evoke left properness):

\begin{theorem}[Hovey]\label{hoveyresult1}
Suppose $\C$ is a pointed simplicial almost finitely generated model category. Suppose further that in $\C$ sequential colimits preserve finite products, and that $Map(S^1,-)$ preserves sequential colimits.

Then stable equivalences in $Sp(\C)$ are detected by the $\tilde{\Omega}^{\infty - n}$ functors.

\end{theorem}

\begin{proof}
This is just a particular case of Theorem 4.9 in \cite{Hovey}.
\end{proof}

\begin{remark}
We shall also need to make use of slightly more general kinds of spectra. Namely, we shall denote by $Sp(\C,A)$ the spectra in $\C$ constructed with respect to the endofunctor $A\wedge \bullet$ (in the language of Hovey). In the particular case that $A=(S^1)^{\wedge n}$ these can be thought of as spectra that only include the $ni$-th spaces.

 Notice that the obvious analog of \ref{hoveyresult1} still holds.

\end{remark}

\subsection{Composing functors with Quillen equivalences}\label{assumpquillenequiv}

Suppose we are given a Quillen equivalence 
\[F\colon \C \rightleftarrows \D \colon G.\]

In this section we show that, up to a zig zag of equivalences of functors, studying homotopy functors into $\C$ is the same as studying homotopy functors into $\D$ (the analogous result for functors from $\C,\D$ also holds, but we won't be using it). This is the content of the following results.

\begin{proposition}
Suppose 
\[F\colon \C \rightleftarrows \D \colon G\]
a Quillen equivalence (between cofibrantly generated model categories). Then $F \circ Q_{\C}$ and $G \circ R_{\D}$ are homotopy inverses, i.e., there are zig zags of w.e.s $G \circ R_{\D} \circ F \circ Q_{\C} \sim id_{\C}$, $F \circ Q_{\C} \circ G \circ R_{\D} \sim id_{\D}$

\end{proposition}

\begin{proof}
By reasons of symmetry we need only to construct the zig zag $G \circ R_{\D} \circ F \circ Q_{\C} \sim id_{\C}$. This is given by the obvious functors
\[id_{\C} \gets Q_{\C} \to G \circ F \circ Q_{\C} \to G \circ R_{\D} \circ F \circ Q_{\C} \]
where the natural transformation $Q_{\C} \to id_{\C}$ is of course given by w.e.s, while the composite natural transformation $Q_{\C} \to G \circ R_{\D} \circ F \circ Q_{\C}$ is given by w.e.s since it is adjoint to the obvious maps\footnote{Recall that a Quillen adjunction is a Quillen equivalence iff a map (with $X$ cofibrant and $Y$ fibrant) $F(X) \to Y$ is a w.e. iff the adjoint map $X \to R(Y)$ is.} $F \circ Q_{\C} \to R_{\D} \circ F \circ Q_{\C}$.

\end{proof}

\begin{corollary}
Suppose given a Quillen equivalence \[F\colon \C \rightleftarrows \D \colon G\]
as before and $(E,W_E)$ a category with w.e.s. Then, when defined, the homotopy categories of homotopy functors $Ho^h(E,\C)$ and $Ho^h(E,\D)$ are equivalent.

\end{corollary}

\section{Goodwillie calculus: polynomial functors}\label{polyfunct}

In this section we define the generalized concepts of {\bf polynomial functor} and construct the universal polynomial approximations $P_n F$ to a homotopy functor $F$. All definition and results are reasonably straightforward generalizations of those of Goodwillie in \cite{Goodwillie}.

\subsection{Cartesian cubes and total homotopy fibers.}

We first set some notation. 

For $S$ a finite set, denote by $\P(S)$ its poset of subsets, and additionally define $\P_i(S) = \{ A \in \P(A) \colon \abs{A} \geq i\}$. Of particular importance will be $\P_1(S)=\P(S) - \{\emptyset\}$.\footnote{This does NOT match Goodwillie's notation.}

Notice that all of these categories are have finite nerves, and hence are both direct and inverse categories (Definition \ref{dirinv}).

\begin{definition}
An {\bf $n$-cube} in a category $C$ is a functor $\X_{\bullet} \colon \mathcal{P}(S)\to C$, where $|S|=n$.

$\X$ is said to be {\bf hocartesian} if the canonical map \[\hocmp(\X) \colon \mathcal{X}_\emptyset \to \holim_{T \in \P_1(S)}\mathcal{X}_T\] is a w.e.\footnote{We now use the traditional notations $\holim$ and $\hocolim$ for the derived functors $\mathbb{R} \varprojlim$ and $\mathbb{L} \varinjlim$ introduced in \ref{cofgen}.}. We call this map the {\bf homotopy comparison map}, and denote it by $\hocmp(\X)$, or just $\hocmp$ when this would not cause confusion.

Furthermore, $\mathcal{X}$ is said to be {\bf strongly hocartesian} if the restrictions $\mathcal{X}|_{\mathcal{P}(T)}$ are hocartesian for any subset $T \subset S$ with $|T| \geq 2$.

There are also obvious dual definitions of {\bf hococartesian} and {\bf strongly hococartesian} $n$-cubes.
\end{definition}

\begin{remark}
Note that hocartesianness is a homotopical property of the cube $\X$, as so is the map $\mathcal{X}_\emptyset \to \holim_{T \in \P_1(S)}\mathcal{X}_T$ (which we regard as a map in $Ho(\C)$). Furthermore, that map can be constructed in the following way: take $\bar{\X}$ an injective fibrant replacement of $\X$. Then $\bar{\X}|_{\P_1(S)}$ is itself injective fibrant, since $\P_1(S)$ is a terminal subposet of $\P(S)$ (see Lemma \ref{restprop}), so that $\lim_{\P_1(S)}\bar{\X}$ is indeed a $\holim$, and the intended map is then given by $\bar{\mathcal{X}}_\emptyset \to \lim_{T \in \P_1(S)}\bar{\mathcal{X}}_T$.
\end{remark}

\begin{remark}
Notice that any strongly hocartesian cube $\X$ is determined by its restriction to $\P_{n-1}(S)$. Indeed, consider $\bar{\X}$ its injective fibrant replacement. Then its restriction $\bar{\X}|_{\P_{n-1}(S)}$ is itself injective fibrant, and it trivially follows that its reextension (i.e., the right Kan extension) $\bar{\X}'$ to $\P(S)$ is itself injective fibrant and the unit map $\bar{\X} \to \bar{\X'}$ is a pointwise equivalence.


\end{remark}

We now turn to the issue of defining the total fiber of a cube, and proving its iterative properties. Suppose $\C$ is now a pointed model category with zero object $\**$.

\begin{definition}
Consider the Quillen adjunction
\[\C \rightleftarrows \C^{\downarrow}\]
(where $\C^{\downarrow}$ has the injective model structure) with left adjoint $X \mapsto (X \to *)$ and right adjoint $(X\to Y) \mapsto \fiber(X \to Y)$.

We then denote by $\hofiber$ the derived functor $\mathbb{R} \fiber \colon Ho(\C^{\downarrow}) \to Ho(\C)$.

More generally, consider the Quillen adjunctions:
\[\C \rightleftarrows \C^{\P(S)}\]
(where $\C^{\P(S)}$ has the injective model structure) with left adjoint $X \mapsto (\emptyset \mapsto X; else \mapsto *)$ and right adjoint $\X \mapsto \fiber(X_\emptyset \to \lim_{T \in \P_1(S)} \X_T)$.

We now denote by $tothofiber$ the derived functor $Ho(\C^{\P(S)}) \to Ho(\C)$.
\end{definition}

\begin{remark}
Though not immediately obvious, our definition of $\hofiber$ does match the obvious alternate candidate definition: $\holim(X \to Y \gets *)$. This follows, for instance, from Proposition A.2.4.4 in \cite{Topos}. That the analogous result also holds for the higher dimensional variants can be proven by combining the following result with an induction argument.

\end{remark}

We now show that the $tothofiber$ can be computed iteratively.

\begin{proposition}{\bf Iterative definition of the total hofiber}\label{itertotfib}

The following diagram commutes up to a natural isomorphism:

\begin{equation}\xymatrix{
Ho(\C^{\P(S\amalg S' \amalg W)}) \ar[dr]_{\hofiber_{S\amalg S'}}\ar[r]^{\hofiber_S} & Ho(\C^{\P(S'\amalg W)}) \ar[d]^{\hofiber_{S'}}\\
& Ho(\C^{\P(W)}) 
}\label{homcom2}\end{equation}

\end{proposition}

\begin{proof}
Notice first that we have obvious Quillen adjunctions (for the injective model structures)
\[\C^{\P(W)} \rightleftarrows \C^{\P(S \amalg W)}\]
with the right adjoint given by $totfiber_S$, the total fiber in the $S$ direction (that this is a Quillen equivalence is clear from looking at the left adjoints).

It then follows that if we compute the paths in \eqref{homcom2} by first performing an injective fibrant replacement, then no further replacements are needed, and one is left with proving that $\fiber_{S \amalg S'} = \fiber_{S'} \circ \fiber_{S}$. But this is obvious: mapping into $\fiber_{S \amalg S'}(\X)$ is the same as giving a map to $\X_\emptyset$ which restricts to the (uniquely determined) zero maps on the other $\X_T$, and $\fiber_{S'} \circ \fiber_{S}(\X)$ clearly satisfies that same universal property.

\end{proof}

The previous result can also be viewed as a consequence/particular case of the following:

\begin{proposition}{\bf Iterative definition of the homotopy comparison map}

The following diagram commutes up to a natural isomorphism:

\begin{equation}\xymatrix{
Ho(\C^{\P(S\amalg S' \amalg W)}) \ar[dr]_{\cmp_{S\amalg S'}}\ar[r]^{\cmp_S} & Ho(\C^{\P(\{\**_S\}\amalg S'\amalg W)}) \ar[d]^{\cmp_{\{\**_S\} \amalg S'}}\\
& Ho(\C^{\P(\{\**_{S \amalg S'}\} \amalg W)}) 
}\label{compmapcom}\end{equation}

\end{proposition}

\begin{proof}
As in the previous proposition we have obvious Quillen adjunctions (for the injective model structures)
\[\C^{\P(\{\**_S\} \amalg W)} \rightleftarrows \C^{\P(S \amalg W)}\]
with the right adjoint given by $\cmp_S$, so again it suffices to compare the results obtained by first performing an injective fibrant replacement and then applying the levelwise constructions. Again it is easy to check that $\cmp_{S \amalg S'} = \cmp_{\{\**_S\}S'} \circ \cmp_{S}$ by appealing to the universal property (this is probably clearer by looking at the left adjoints, which are manifestly equal).
\end{proof}


\begin{definition}
Let $X$ be any object in $\C$ a {\bf pointed simplicial } cofibrantly generated model category. Recall that $\C$ is then a $SSets_{\**}$-model category.

Define $C(X)= X \wedge (\Delta^1,\{0\})$.

There is a standard map $X \xrightarrow{c_X} C(X)$ induced by $(\{0,1\},\{0\}) \to (\Delta^1,\{0\})$.
\end{definition}



\begin{remark}

Note that, when $X$ is cofibrant, then $C(X) \sim \**$, and the standard map $X \to C(X)$ is a cofibration.
\end{remark}

\begin{definition}
\label{X X}
Denote by $\X_X$ the left Kan extension of the diagram
\begin{equation}\xymatrix{& & X \ar[drr]\ar[dr] \ar[dl] \ar[dll]& & \\ C(X) & C(X) & \dotsc & C(X) & C(X) \\
}\label{X cube}\end{equation}

\end{definition}

\begin{remark}

Notice that, when $X$ is cofibrant, then \eqref{X cube} is a cofibrant $\P_{\leq 1}(S)$ diagram, so that $\X_X$ is a strongly cocartesian cube.

More generally, when $X$ not cofibrant, we will also think of $\X_{Q(X)}$ as the strongly cartesian cube associated to $X$.

\end{remark}

\begin{proposition}\label{comcones}
Suppose that either 

The functors $\P(S) \times \P(S) \to \C$ given by $(U,V) \mapsto \X_{\X_X(U)}(V)$ and $(U,V) \mapsto \X_{\X_X(V)}(U)$ are naturally isomorphic. Furthermore, this natural isomorphism is also natural in $X$.
\end{proposition}

\begin{proof}

Since $C(X)$ commutes with colimits, and $\X_\bullet$ is constructed via a left Kan extension of its values on $\P_{\leq 1}$, the result will follow if we prove it for the restriction of the functors to $\P_{\leq 1}(S) \times \P_{\leq 1}(S)$.

This then amounts to producing an automorphism of $C(C(X))$ which swaps $c_{C(X)}$ and $C(c_X)$. But $C(C(X))$  is just $(X\wedge (\Delta^1,\{0\}))\wedge (\Delta^1,\{0\})$, and the required maps induced by the inclusions $(\{0,1\},\{0\}) \to (\Delta^1,\{0\})$ into each of the product terms, so the required automorphism just follows from the symmetry isomorphisms for the monoidal structure in $SSet_{\**}$.

\end{proof}





\subsection{Excisive functors}

We can now finally define polynomial/excisive functors.

\begin{definition}
A homotopical functor $F \colon \C \to \D$ is said to be {\bf $d$-excisive} if for any $\X$ a strongly hococartesian $(d+1)$-cube in $\mathcal{C}$ we have that $F(\X)$ is a hocartesian cube in $\D$.
\end{definition}

\begin{remark}
Notice that, since hocartesianness and strong hococartesianness of cubes are homotopical properties, the excisiveness of an homotopy functor $F$ is entirely determined by looking at the functors
\[Ho(\C^{\P(S)}) \xrightarrow{Ho(F^{\P(S)})} Ho(\D^{\P(S)})\]
which we denote by $hF_S$, for short.
\end{remark}

We know prove some basic properties of excisive functors.

\begin{proposition}
If $F$ is $d$-excisive, then it is also $d'$-excisive for $d' \geq d$.
\end{proposition}

\begin{proof}
It suffices to prove that $d$-excisiveness implies $d+1$-excisiveness. 
Write a set $\bar{S}= S \amalg \{ \** \}$ for a set of cardinality $d+1$, so that an $\bar{S}$-cube $\X$ can be viewed as a map of $S$-cubes $\X^b \to \X^t$. By definition $\X^b$ and $\X^t$ are strongly hococartesian, so that by hypothesis $F(\X^b)$ and $F(\X^t)$ are cartesian. But then so must be $\X$ by applying Proposition \ref{compmapcom}. 
\end{proof}

\begin{definition}
Suppose $\C$ is pointed. A sequence of functors $F\to G \to H$ is a hofiber sequence if the composite is the $\**$ functor, and $F$ is pointwise the homotopy fiber of $G \to H$. Or in other words, if the squares
\begin{equation}
\xymatrix{
F(c) \ar[d]\ar[r] & \** \ar[d]\\
G(c) \ar[r] & H(c)
}\label{hofiber}
\end{equation}
are hocartesian.
\end{definition}

\begin{remark}
Given a hofiber sequence $F\to G \to H$, consider 
\begin{equation}
\xymatrix{
F \ar[d]\ar[r]& G \ar[d]\ar[r]& H \ar[d]\\
\bar{F} \ar[r]& \bar{G} \ar[r]& \bar{H}}
\end{equation}
where $\bar{G} \to \bar{H}$ is a fibrant replacement (pointwise in the arrow category) of $G \to H$ and $\bar{F}=\fiber(\bar{G}\to \bar{H})$. It then follows it must be $F \xrightarrow{\sim} \bar{F}$, so that one concludes that the hofiber of a map of functors is well defined up to a zig zag of natural w.e.s.
\end{remark}

\begin{proposition}
If $F\to G \to H$ is a hofiber sequence of homotopy functors with $G$ and $H$ both $d$-excisive, then so is $F$.
\end{proposition}

\begin{proof}
Given a $S$-cube $\X$, form the $S \amalg \{ \** \}$-cube $G(\X) \to H(\X)$.

Now consider the following diagram:

\begin{equation}\xymatrix{
Ho(\D^{\P(S \amalg \{ \** \})}) \ar[d]^{\hofiber_*}\ar[r]^{\cmp_S} & Ho(\D^{\P(\{\**_S, \** \})}) \ar[d]^{\hofiber_*}\\
Ho(\D^{\P(S)}) \ar[r]^{\cmp_S} & Ho(\D^{\P(\{ \**_S\})}) 
}\end{equation}

By repeating the techniques of the proofs of Propositions \ref{homcom2} and \ref{compmapcom} this diagram commutes up to natural isomorphism: namely, it suffices to perform the constructions by first choosing an injective fibrant replacement and then checking that $\cmp_S \circ \fiber_* = \fiber_* \circ \cmp_S$.

What we want to prove is then that $\cmp_S{F(\X)}$ is a w.e.. Since $F(\X)=\hofiber_*(G(\X) \to H(\X))$, this follows if we show $\hofiber_*(\cmp_S(G(\X)) \to H(\X))$ is a w.e.. But this is clear since, by the $d$-excisiveness hypothesis, $\cmp_S(G(\X))$ is a square for which two opposing arrows are w.e.s.
\end{proof}

\begin{lemma}\label{partexc}
Suppose $F$ is $d$-excisive, and that $\X_\bullet$ is a strongly hocartesian $e$-cube, with $e \geq d$. Then the map 
\[F(\X_\emptyset) \to \holim_{T \in\P_{e-d+1}} F(\mathcal{X}_T)\]
is a w.e.. 
\end{lemma}

\begin{proof}
This being an homotopical property, we are free to replace $F(\X)$ by an injective fibrant replacement $\mathcal{Y}$, whose sub-cubes of dimension $\geq d$ are all hocartesian (by the previous result).

Now let $\bar{\mathcal{Y}}$ be the right Kan extension of the restriction $\mathcal{Y}|_{\P_{e-d+1}}$. $\bar{\mathcal{Y}}$ is still fibrant since both adjoints preserve fibrant objects (by Proposition \ref{restprop}). The unit map $\mathcal{Y} \to \bar{\mathcal{Y}}$ is tautologically a w.e. on the points in $\P_{e-d+1}$, but by induction it is also so on all other points, and the remaining maps can all be identified with $\cmp_S$ for some $S$ with $\abs{S}\geq d$.
But the map we want to check is a w.e. is precisely $\mathcal{Y}_\emptyset \to \bar{\mathcal{Y}}_\emptyset$, hence we are done.
\end{proof}

\begin{lemma}\label{prevlemma}
Let $A$ be a category covered by subcategories $\{A_s\}_{s \in S}$(i.e., any arrow in $A$ lies in some $A_s$. Note the $A_s$ are not assumed luff\footnote{A subcategory of $\C$ is called luff if it has the same set of objects.}), and suppose given a functor $F \colon A\to \mathcal{D}$.

Define an $S$-cube $\X$ by $T \mapsto \lim (F|_{\bigcap_{s \in T}(A_s)})$ (the empty intersection is taken to be all of $A$). Then $\X$ is cartesian.
\end{lemma}

\begin{proof}
This is simply a matter of diagram chasing to check that the universal properties of the terms in $\cmp(\X)$ match.

\end{proof}


\begin{proposition}\label{Good3.4} 

Let $L\colon \mathcal{C}^d \to \mathcal{D}$ be a homotopy functor which is 1-excisive (separately) in each variable, and denote by $\Delta \colon \C \to \C^d$ the diagonal functor. Then $L \circ \Delta \colon \mathcal{C} \to \mathcal{D}$ is $d$-excisive. 
\end{proposition}

\begin{proof}
Consider $\X$ any strongly hococartesian $S$-cube, with $\abs{S}=d+1$.

Denote by $L(\X)^d$ the composite $\P(S)^d \xrightarrow{\X^{\times d}} \C^d \xrightarrow{L} \D$. We wish to show that $L(\X)^d \circ \Delta$ is a hocartesian cube, but by homotopy invariance this can equivalently be done after replacing $L(\X)^d$ by an injective fibrant multicube $\mathcal{Y}$. By retracing the steps in the proof of Lemma \ref{partexc} (and applying them inductively), one concludes that one might as well further replace $\mathcal{Y}$ by the right Kan extension $\bar{\mathcal{Y}}$ of its restriction to $\P_{n-1}(S)^d$.

We know apply Lemma \ref{prevlemma} to $\P_{n-1}(S)^d$: take $A_s = \{(T_1,\cdots,T_d) \in \P_{n-1}(S)^d \colon \underset{i}{\forall} s \in T_i\}$, or equivalently, $A_s$ is the undercategory $(\{s\},\cdots,\{s\}) \downarrow \P_{n-1}(S)^d $. It is easily seen that intersections $\bigcap_{s \in T}(A_s)$ with $T$ non empty are the undercategories $(T,\cdots, T) \downarrow \P_{n-1}(S)^d$, and that, by a pigeonhole argument $\bigcup_{s \in S}A_s = \P_{n-1}(S)^d$.

Hence the cube constructed by Lemma \ref{prevlemma} is precisely $\bar{\mathcal{Y}} \circ \Delta$, and we will be done provided this cube is also fibrant (so that the $limits$ are indeed $holimits$). This follows by checking that the adjunction $\D^{\P(S)} \rightleftarrows \D^{\P(S)^d}$ is Quillen, and this is in turn clear at the level of left adjoints since $\Delta \colon \P(S) \to \P(S)^d$ is inclusion of sublattices, so that $Lan \mathcal{H}(y) = \mathcal{H}(max(x \in \P(S) \colon x \leq y))$.
\end{proof}

\begin{remark}
Essentially the same proof applies to the case where $L$ is known to be $n_i\geq 1$ excisive in each variable, the result being that $L \circ \Delta$ is $n=n_1+\dots+n_d$ excisive. In that version of the proof $\P_{n-1}(S)^d$ is replaced by $\P_{n-n_1}(S) \times \dotsm \times \P_{n-n_d}(S)$, everything else following similarly.
\end{remark}


\section{The Taylor tower}\label{taylor}

We now turn to the task of defining the universal $n$-excisive approximations $P_nF$ to a homotopy functor $F$. Rather than receive a straight up functor $F \to P_n F$, we will have a zig zag $F \rightsquigarrow P_nF$ (with backward arrows being pointwise w.e.s). With this in mind, and to simplify notation, we assume from now on that $F$ takes values in cofibrant objects, that is to say, we assume $F$ is replaced by $Q_{\D} \circ F$ if necessary.

\begin{definition}
Assume that $\C$ is a pointed simplicial cofibrantly generated model category, and that $\D$ is a cofibrantly generated model category. Let $\X_X$ denote the  cubes of Definition \ref{X X}.

Define $\tilde{\bar{t}}_nF \colon F \to \tilde{\bar{T}}_n F = \holim_{T \in \P_1(S)} \X_{\bullet}(T)$ as the standard map to the $\holim$, and 
$\bar{t}_nF \colon F \rightarrowtail \bar{T}_n F$ as $Q(\tilde{\bar{t}}_nF)$ the {\bf canonical cofibrant factor} of $\tilde{\bar{t}}_nF$.

We then define $t_nF \colon F \rightsquigarrow T_n F$ as the zig zag $F \overset{\sim}{\gets} F \circ Q \overset{\bar{t}_nF \circ Q}{\rightarrowtail} \bar{T}_nF \circ Q = T_nF.$

\end{definition}

\begin{remark}
Notice that the $\bar{T}_nF$ defined above is not a homotopy functor, since $\X_\bullet$ is not itself an homotopical construction. It is, however, a ``left Quillen'' functor, in the sense that $\bar{T}_nF \circ Q$ is now a homotopy functor, since so is $\X_{Q(\bullet)}$.

\end{remark}

The need for the $\bar{T}_nF$ construction is the following: ideally, one would like to just define $P_nF = \hocolim_\alpha T_n^\alpha F$. However, obtaining functorial cartesian cubes from an object $X$ seems to require making a cofibrant replacement first. Hence, so as not to make infinitely many cofibrant replacements when constructing the intended $P_nF$, we merely do it once and then iterate the intermediate $\bar{T}_nF$ construction instead.

\begin{definition}
Let $\lambda$ be an ordinal. 
Define, by transfinite induction, 
\begin{equation}
\begin{split}
\bar{T}_n^\lambda F &= \bar{T}_n (\bar{T}_n^{\tilde{\lambda}}F) \text{, if $\lambda = \tilde{\lambda}+1$, a successor ordinal}   \\
&= \colim_{\beta < \lambda} \bar{T}_n^{\beta} F \text{, if $\lambda$ a limit ordinal.} \\
\end{split}
\end{equation}

\end{definition}

\begin{definition}
Let $\kappa$ be the chosen regular cardinal with respect to which the target category $\D$ is cofibrantly generated.

Set $\bar{P}_nF = T^{\kappa}_n F$, and denote by $\bar{p}_nF \colon F \to \bar{P}_nF$ the standard map.

We then define $p_nF \colon F \rightsquigarrow P_n F$ as the zig zag $F \overset{\sim}{\gets} F \circ Q \overset{\bar{p}_nF \circ Q}{\rightarrowtail} \bar{P}_nF \circ Q = P_nF.$

\end{definition}

\begin{remark}
Again notice that, just like in the case of $\bar{T}_nF$, the functor $\bar{P}_nF$ is not in general a homotopy functor. And, likewise, note that $P_n F$ is a homotopy functor, since so is $T_nF$ and the $colimit$s used in the construction are actually $\hocolim$s (this is the reason to demand the map $F \rightarrowtail \bar{T}_n F$ be a cofibration and that $F$ be pointwise cofibrant).

\end{remark}

We now turn to the task of proving that $F \overset{p_n F}{\rightsquigarrow} P_n F$ is the universal zig zag from $F$ to an $n$-excisive functor. First we check that indeed $P_n F$ is $n$-excisive.

\begin{lemma}\label{factor}
Suppose $\mathcal{W}_\bullet$ a cofibrant strongly hococartesian cube that is the left Kan extension of its restriction to $\P_{\leq 1}(S)$.

Then the map of cubes $F(\mathcal{W}) \xrightarrow{\bar{t}_n \circ \mathcal{W}} \bar{T}_n F(\mathcal{W})$ factors through a hocartesian cube.
\end{lemma}

\begin{proof}

First notice that one needs only prove the result for $F \circ \mathcal{W} \xrightarrow{\tilde{\bar{t}}_nF \circ \mathcal{W}} \tilde{\bar{T}}_n F \circ \mathcal{W}$, as then the result for $\bar{T}_n F(\mathcal{W})$ follows by applying $Q$, the cofibrant factor functor.

Now consider the following $\P(S) \times \P_{\leq 1}(S)$ diagrams in $\C$:

\begin{equation}
\begin{split}
\mathcal{Y}_{U,V}= & \mathcal{W}_U \text{ if } V=\emptyset\\
& \colim(\mathcal{W}_U \otimes \Delta^1 \gets \mathcal{W}_U \otimes \{0\} \to \mathcal{W}_{U \cup V}) \text{ otherwise}\\
\mathcal{Z}_{U,V}= & \mathcal{W}_U \text{ if } V=\emptyset\\
& C(\mathcal{W}_U) \text{ otherwise}
\end{split}
\end{equation}
where maps in the $U$ direction are obvious, and those in the $V$ direction are induced by $\{1\} \to \Delta^1$. Furthermore, unwinding the definition of $C(\mathcal{W}_U)=\mathcal{W}_U \wedge (\Delta^1,\{0\})= \colim(\mathcal{W}_U \otimes \Delta^1 \gets \mathcal{W}_U \otimes \{0\} \to \**)$, we see there is an obvious functor $\mathcal{Y} \to \mathcal{Z}$. Notice also that, since the values of $\mathcal{W}_\bullet$ are all cofibrant, the diagrams $\mathcal{Y}_{U,\bullet}$ are $\P_{\leq 1}(S)$ cofibrant, and that there is an obvious map $\mathcal{Y}_{U,\bullet} \to \mathcal{W}_{U \cup \bullet}$, where $\mathcal{W}_{U \cup \bullet}$ is itself $\P_{\leq 1}(S)$ cofibrant.

Now let $\bar{\mathcal{Y}}, \bar{\mathcal{Z}}$ be the $\P(S) \times \P(S)$ obtained by doing the left Kan extensions along the second variable. Notice that $\bar{\mathcal{Z}}_{U,V} = \X_{\mathcal{W}_U}(V)$. By the above remarks about cofibrancy, we have a map $\bar{\mathcal{Y}}_{U,V} \to \mathcal{W}_{U \cup V}$ which is a pointwise equivalence (that the left Kan extension of the restriction $\mathcal{W}_{U \cup V}|_{\P(S) \times \P_{\leq 1}(S)}$ is again $\mathcal{W}_{U \cup V}$ is where we use that $\mathcal{W}$ was itself given by a left Kan extension). 

The map we were trying to factor is precisely the full composite:
\[F(\mathcal{W}_\bullet) \to  \holim_{V \in \P_1(S)}F(\bar{\mathcal{Y}}_{\bullet,V}) \to \holim_{V \in \P_1(S)}F(\bar{\mathcal{Z}}_{\bullet,V})\]
so we will be done if we prove $\holim_{V \in \P_1(S)}F(\bar{\mathcal{Y}}_{\bullet,V})$ a cartesian cube.

But
comparison maps commute with $\holim_{V \in \P_1(S)}$ (as holims are readily seen to commute with each other in the presence of injective model structures), so that it is enough to show that the $F(\bar{\mathcal{Y}}_{\bullet,V})$ are hocartesian cubes for each $V \neq \emptyset$, which by homotopy invariance is the same as showing that the $F(\mathcal{W}_{\bullet \cup V})$ are hocartesian, which is now clear since in those cubes the maps in the $V$ directions are all w.e.s (and in fact isomorphisms).

\end{proof}

\begin{corollary}
For $F$ any homotopy functor, $P_nF$ is $n$-excisive.
\end{corollary}

\begin{proof}
First notice that, since $P_nF$ is a homotopy functor, we need only check that the cubes $\mathcal{W}$ described in Lemma \ref{factor} are sent to hocartesian cubes. Furthermore, since those cubes are pointwise cofibrant, we might as well just check hocartesianness for $\bar{P}_nF \circ \mathcal{W}$.

But the result is now obvious from Lemma \ref{factor} and Proposition \ref{comholim}, as the latter shows that comparison maps commute with the $\hocolim$s in the $\bar{P}_n$ construction, and, by cofinality, those $\hocolim$s can be taken over the factoring strongly hocartesian cubes.

\end{proof}

In order to check universality we'll also need the following:

\begin{proposition}\label{baspnprop}
$\phantom{1}$

\begin{itemize}
\item Let $a \mapsto F_a$, where $a \in A$ a finite category, be a compatible family of homotopy functors, and let $F$ be another homotopy functor equipped with compatible natural transformations $F \to F_a$ displaying $F = \holim_A F_a$. Then the natural transformations $P_nF \to P_nF_a$ also display $P_nF = \holim_A P_nF_a$.
\item Similarly, let $\beta \mapsto F_{\beta}$, for $\beta < \kappa$, a transfinite sequence of homotopy functors ($\kappa$ assumed as in Proposition \ref{comholim}), and let $F$ be another homotopy functor equipped with compatible natural transformations $F_{\beta} \to F$ displaying $F = \hocolim_{\beta} F_{\beta}$. Then $P_nF_{\beta} \to P_nF$ display $P_nF = \hocolim_{\beta} P_nF_{\beta}$.
\end{itemize}
\end{proposition}

\begin{proof}
For the first part, commutativity of $\holim$s with each other immediately gives that $T_nF = \holim_A T_nF_a$, so that the result then follows by applying  transfinite induction and Proposition \ref{comholim}.

For the second part, first apply Proposition \ref{comholim} to conclude the result for $T_nF = \hocolim_{\beta} T_nF_{\beta}$, and then transfinite induction plus the fact that $\hocolim$s commute with each other.
\end{proof}

The first part of the previous result has the following consequence.

\begin{corollary}\label{pncmp}
Let $F_T$, $T \in \P(S)$, be a cube of functors. Then $P_n(\cmp(F_{\bullet}))$ is w.e. to $\cmp(P_nF_{\bullet})$
\end{corollary}

\begin{theorem}\label{univ}
Assume that the categories $\C, \D$ are cofibrantly generated. Assume further that $\C$ is pointed simplicial, or a left Bousfield localization of a pointed simplicial model category.

Then $F \overset{p_n F}{\rightsquigarrow} P_n F$ exhibits $P_n F$ as the universal $n$-excisive approximation to $F$.
\end{theorem}

\begin{proof}

We first prove existence of factorization. Given a zig zag $f \colon F \rightsquigarrow R$, with $R$ assumed $n$-excisive, first replace it by $f \circ Q \colon F \circ Q \rightsquigarrow R \circ Q$. 

Given a weak map $F(Q)\to R$, where $R$ is assumed $n$-excisive, just consider the commutative diagram of weak maps:

\begin{equation}
\xymatrix{
F \circ Q \ar[d]_{\bar{p}_nF \circ Q}\ar@{~>}[r]^{f \circ Q} & R \circ Q \ar[d]^{\bar{p}_nR \circ Q}\\
P_nF\ar@{~>}[r]^{P_nf} & P_nR 
}
\end{equation}

Since $\bar{p}_nR \circ Q$ is clearly an equivalence when $R$ is $n$-excisive (the $\bar{t}_n$ where constructed to ensure this, and $\bar{p}_n$ is merely their transfinite composition), we have that the zig zag $\P_nF \rightsquigarrow P_nR \overset{\sim}{\gets} R \circ Q$ provides a factorization $\tilde{f}$ of $f \circ Q$ through $\bar{p}_nF \circ Q$.

We now prove uniqueness of the factorization in the homotopy sense, i.e., that any two factorizations can be related by w.e.s.

Consider the diagram

\begin{equation}
\xymatrix{
F \ar@{~>}[d] \ar@{~>}[r]^{p_nF} & P_nF \ar@{~>}[d]^{\sim}\ar@{~>}[r]^{\tilde{f}'} & R \ar@{~>}[d]^{\sim} \\
P_nF\ar@{~>}[r]^{P_n(p_nF)} & P_n(P_nF) \ar@{~>}[r]^{P_n(\tilde{f}')} & P_nR 
}
\end{equation}

The factorization of $\tilde{f}' \circ p_nF$ constructed in the first part of the proof is given by the ``lower followed by right'' zig zag in the diagram. Which, given two of the vertical maps are known to be w.e.s, will follow if $P_n(p_nF)$ is a zig zag of w.e.s. 

This amounts to showing that $P_n(\bar{p}_nF \circ Q)$ is a w.e., and this, by the second part of Proposition \ref{baspnprop}, will follow if we show $P_n(\bar{t}_nF \circ Q)$ a w.e., which by Corollary \ref{pncmp} amounts to showing that $P_n(F\circ \X_{Q(\bullet)})$ is a hocartesian cube of functors. And this is in turn is equivalent to showing that $\bar{P}_n(F \circ \X_{\bullet})$ gives a cartesian cube when applied to cofibrant objects.

But now notice that the cube $\bar{P}_n(F \circ \X_{\bullet})$ is isomorphic to the cube $\bar{P}_nF \circ \X_{\bullet}$: indeed, that this is true at the $\bar{T}_n$ level follows from Proposition \ref{comcones} (which actually establishes such isomorphisms for the $\bar{t}_n$ functor itself), and the claim for $\bar{P}_n$ follows by induction (with the additional remark that precomposing with $\X_\bullet$ obviously commutes with limits of functors).

But now one finally concludes hocartesianness of  $\bar{P}_nF \circ \X_{\bullet}$ as an instance of $P_nF $ being $n$-excisive, and the proof is complete.

\end{proof}

\begin{remark}\label{extracond}
Much of the above follows just fine using alternate functorial definitions of the ``cones'' $X \rightarrowtail C(X)$. However, such cones typically only satisfy a weak version of Proposition \ref{comcones} (indeed, in the previous Theorem the extra conditions on $\C$ are merely to guarantee that proposition holds) and, as such, it seems technically hard to produce the corresponding ``uniqueness'' part of the previous proof.

\end{remark}

\section{Goodwillie calculus and stable categories}\label{goodstab}

Our ultimate goal in this section will be to adapt the proof of the following result of Goodwillie's: let $F \colon \C \to \D$ be a $d$-homogeneous homotopy functor, whose target category $\D$ is stable, and let $cr_d F \colon \C^d \to \D$ be its cross effect. Then $F(X) = cr_d F(X,\cdots,X)_{h \Sigma_d}$.

\subsection{Hofibers/hocofibers detect weak equivalences}

We start by proving some basic results on stable categories. We recall the definition:

\begin{definition}
Let $\D$ be a pointed simplicial (cofibrantly generated) model category. Then $\D$ is said to be {\bf stable} if the Quillen adjunction
\[(S^1,\**)\wedge \bullet=\Sigma \colon \D \rightleftarrows \D \colon \Omega = Map((S^1,\**),\bullet)\]
is actually a Quillen equivalence.

We will denote by $\bar{\Sigma}=\Sigma \circ Q$, $\bar{\Omega}= \Omega \circ F$ chosen derived functors.
\end{definition}

\begin{remark}
Note that this definition can also be made non simplicially, though we will not be using it in such generality.
\end{remark}

\begin{proposition}
Let $\D$ be a stable model category. Then every hococartesian square is also a hocartesian square.
\end{proposition}

\begin{proof}
Following the definition of the Goodwillie tower in the last section we have that $T_2 Id \sim \bar{\Omega} \bar{\Sigma}$ with the map $Id \rightsquigarrow T_2 Id $ that induced by the Quillen equivalence. Hence $P_2Id \sim Id$, and the result follows.
\end{proof}

\begin{proposition}\label{contrhocof}
Let $\D$ be stable model category, and consider a map $X \xrightarrow{f} Y$. Then $F$ is a w.e. iff the hocofiber $hocof(f)$ is contractible.
\end{proposition}

\begin{proof}
Without loss of generality assume $f$ a cofibrant diagram. Then the diagram
\begin{equation}
\xymatrix{ X \ar[r]^f \ar[d] & Y \ar[d]\\
\** \ar[r] & hocof(f)
}
\end{equation}
is a homotopy pushout and hence, by the previous proposition, also a homotopy pullback, and the result is now obvious.
\end{proof}

\begin{proposition}\label{doublefib}
Let $A \xrightarrow{f} B \xrightarrow{g} C$ be a sequence of maps with $g \circ f = *$. Then, if the sequence is a hofiber sequence we have $hof(f) \sim \bar{\Omega}C$. Conversely, if the sequence is a hocofiber sequence, we have $hocof(g) \sim \bar{\Sigma} A$.
\end{proposition}

\begin{proof}
This is clear from the simplicial models for hocofibers/ hofibers.
\end{proof}

\begin{corollary}\label{basecase}
Let $\D$ be a stable category, and $X \xrightarrow{f} Y$ a map. Then $\hofiber (f) \sim \bar{\Omega} \hocofiber(f)$.
\end{corollary}

\begin{proof}
This is clear from the proof of Proposition \ref{contrhocof} and Proposition \ref{doublefib}.
\end{proof}

\begin{corollary}
Let $\D$ be stable model category, and consider a map $X \xrightarrow{f} Y$. Then $F$ is a w.e. iff the hocofiber $\hocof(f)$ is contractible.
\end{corollary}

\begin{proof}
Obvious from Corollary \ref{basecase} and Proposition \ref{contrhocof}.
\end{proof}

\begin{corollary}\label{fibcofib}
Let $\D$ be stable model category.

Let $\X \colon \P(S) \to \D$ be a $d$-cube. Then $hototfiber{\X} \sim \bar{\Omega}^d hototcofiber{\X}$.
\end{corollary}

\begin{proof}

By using Proposition \ref{itertotfib} and its dual (concerning cofibers), one sees that the result can be proven inductively. But then Proposition \ref{basecase} and its ``relative case'' immediately provide both the base case and the induction step.

\end{proof}

\begin{proposition}\label{fibfib}
Let $\D$ be a pointed model category, and $A \xrightarrow{f} B \xrightarrow{g} C$ a sequence of maps. Then $\hofiber(f) = \hofiber (\hofiber(g \circ f) \to \hofiber(g))$.
\end{proposition}

\begin{proof}
Clear from first choosing fibrant diagrams.
\end{proof}

\begin{proposition}\label{basecase2}
Let $\D$ be a stable model category, and $A \xrightarrow{f} B \xrightarrow{g} A$ a sequence of maps with $g \circ f = id_A$. Then $hocofiber(f) \sim \hofiber (g)$ and, furthermore, this equivalence is obtained from the natural sequence of maps $\hofiber(g) \to B \to hocofiber (f)$.

\end{proposition}

\begin{proof}
Consider the obvious maps $A \vee \hofiber(g) \to B \to A$. Proposition \ref{fibfib} shows the first map to be an equivalence, so that $B \sim A \vee \hofiber(g)$ with the maps appearing in the proposition being the standard inclusions/projections, so that the result is now obvious.
\end{proof}

\begin{remark}
Both of the previous results can be proven in a functorial sense, i.e., with the diagrams depending functorially on some other category. This allows one to prove the result that follows.

\end{remark}

\begin{corollary}\label{doublecube}

Let $\D$ be a stable model category, and let $\mathbb{Z} \colon (0 \to 1 \to 2)^{\times d} \to \D$ be a ``double cube'' which is idempotent in each direction. Then
\begin{equation}
hototcofiber(\mathcal{Z}|_{(0 \to 1)^{\times d}}) \sim hototfiber(\mathcal{Z}|_{(1 \to 2)^{\times d}})
\end{equation}
Furthermore, this equivalence is exhibited by the natural composition \[hototfiber(\mathcal{Z}|_{(1 \to 2)^{\times d}}) \to \mathcal{Z}(1,\cdots,1) \to hototcofiber(\mathcal{Z}|_{(0 \to 1)^{\times d}})\].
\end{corollary}

\begin{proof}
This follows inductively by combining Proposition \ref{basecase2} and Proposition \ref{itertotfib}.

\end{proof}

\subsection{Cross-effects}

Let $\C$ be a pointed simplicial cofibrantly generated model category. 

\begin{definition}
Given a $d$-tuple of objects $X_1,\cdots,X_d$, define a cube \[\X^{cr}_{X_1,\cdots,X_d}\colon \P(S) \to \C\] by
\begin{equation}
\X^{cr}_{X_1,\cdots,X_d}(T) = \bigvee_{i \in S - T} X_i \vee \bigvee_{j \in T} C(X_j)
\end{equation}

\end{definition}

\begin{remark}\label{crossexc}
Note that when all of the $X_i$ are cofibrant the cube obtained is $\P(S)$-projective cofibrant, and that it is, in fact, the left Kan extension of its restriction to $\P_{\leq 1}(S)$, so that it is actually a strongly hococartesian cube.
\end{remark}

\begin{definition}
Let $F \colon \C \to \D$ a homotopy functor.

We set
\begin{equation}
\bar{cr}_n F (X_1, \cdots, X_d)= hototfiber(F \circ \X^{cr}_{X_1,\cdots,X_d})
\end{equation}

and define the cross effect $cr_n F = \bar{cr}_n F \circ Q^{\times d}$.

\end{definition}

\begin{lemma} \label{crossdetect}
Suppose $\D$ a stable model category, and let $F \colon \C \to \D$ be a $d$-excisive functor. Then $F$ is also $d-1$-excisive iff $cr_n F \sim \**$.
\end{lemma}

\begin{proof}
The ``if'' direction is obvious from Remark \ref{crossexc}.

For the ``only if'' direction, we must show that, if $cr_n F \sim \**$, then $F$ sends any strongly hococartesian cube $\X$ (which can as always be assumed the left Kan extension of a cofibrant restriction to $\P_{\leq 1} (S) $) to a hocartesian cube $F(\X)$. Constructing the cube $\X'(T) = \colim (C(\X(\emptyset)) \gets \X(\emptyset) \to \X(T))$, we immediately obtain a strongly hocartesian $d+1$-cube $\X \to \X'$. Hence, by $\ref{itertotfib}$ and the assumption that $F$ is $d$-excisive, it will suffice to show that $X'$ itself is sent to a hocartesian cube, and we have hence reduced to the case of cubes with $\X(\emptyset) \sim \**$ (which can be replaced by cubes were it is actually $\X(\emptyset)$).

Such a cube $\X$ (again assumed left Kan extension of cofibrant restriction) is totally determined by specifying $d$ cofibrant objects $X_1, \cdots, X_d$. Given those, consider the ``double cube'' $\mathcal{Z}\colon (0 \to 1 \to 2)^{\times d} \to \D$ given by ``wedging'' the one dimensional ``double cubes'' $\** \to X_i \to C(X_i)$. 

We now apply Proposition \ref{doublecube} to $\mathcal{Z}$; $\mathcal{Z}|_{(1 \to 2)^{\times d}}$ is the cube appearing in the definition of $cr_n F$; while $\mathcal{Z}|_{(0 \to 1)^{\times d}}$ is the cube $\X$. Hence $hototcofiber (\X) \sim \**$, which by \ref{fibcofib} also implies $hototfiber (\X) \sim \**$, as desired.
\end{proof}

\begin{definition}
Now let $L \colon \C^{n} \to \D$ be a symmetric multi-homotopy functor (see Section \ref{factspectra} for the definition). Set $\Delta_n L \colon \C \to \D$ by $\Delta_n L(X) = L(X,\cdots,X)_{h \Sigma_n}$.
\end{definition}

(Here, has usual, we use functorially chosen cofibrant  replacements to construct the hoorbits)

We thus now have constructions $cr_n$ transforming a homotopy functor into a symmetric $n$-homotopy functor, and $\Delta_n$ doing the reverse.

The following results (all of which are adapted from section 3 of \cite{Goodwillie}) show that, when the target category $\D$ is stable, these are ``inverse constructions'' when restricted to $n$-homogeneous functors and $n$-symmetric multilinear functors, respectively. 

Indeed, that $\Delta_n$ takes multilinear functors to $n$-excisive functors is Proposition \ref{Good3.4} (The usage of hoorbits is inconsequential since the target category $\D$ is assumed stable). That $L \circ \Delta_n$ is then also $n$-reduced (and hence homogeneous), i.e., $P_{n-1}F \sim \**$, follows from the following two propositions:

\begin{proposition}
If $\C^n \xrightarrow{L} \D$ is a $(1,\cdots,1)$-reduced homotopy functor, then for any $X$ cofibrant, the map
\[ (L \circ \Delta)(X) \xrightarrow{\bar{t}_{n-1}(L \circ \Delta)} \bar{T}_{n-1}(L \circ \Delta)(X)\]
factors through a contractible object.
\end{proposition}

\begin{proof}
Consider the the maps
\begin{equation}
\begin{split}L(X,\cdots,X) & \to \holim_{\P_1(S)^{\times n}}L(\X_X(U_1),\cdots,\X_X(U_n)) \\
& \to \holim_{\mathcal{E}}L(\X_X(U_1),\cdots,\X_X(U_n)) \\
& \to \holim_{\P_1(S)}L(\X_X(U_1),\cdots,\X_X(U_n))
\end{split}
\end{equation}

where $\mathcal{E}$ is the subset of $\P_1(S)^{\times n}$ such that for some $s$ one has $s \in U_s$, and $\P_1(S)$ denotes the ``diagonal'' subset of $\P_1(S)^{\times n}$.

Notice that$\P_1(S)^{\times n} \supset \mathcal{E} \supset \P_1(S)$ so that the maps are just obtained by restriction.

We will now be done if we can show that $\holim_{\mathcal{E}}L(\X_X(U_1),\cdots,\X_X(U_n))$ is contractible. To see this, we prove that this holim is equivalent to the holim over $\mathcal{E}^*$, the subset of $\P_1(S)^{\times n}$ such that for some $s$ one has $U_s = \{s\}$. 

We show homotopy initiality of $\mathcal{E}^*$ in $\mathcal{E}$. In order words (using Theorem 8.5.5 from \cite{Emily}), we need to show that for every $(U_1,\cdots,U_s) \in \mathcal{E}$, $\mathcal{E}^* \downarrow (U_1,\cdots,U_s)$ is contractible after realization. But, setting $H = \{s \colon s \in U_s\}$, this is a union $\bigcup_{h \in H}\abs{\prod_{s \neq h} \P_1(U_s)}$ of contractible subspaces with contractible intersections ($\bigcap_{h \in H'}\abs{\prod_{s \neq h} \P_1(U_s)} \simeq \abs{\prod_{s \notin H'} \P_1(U_s)}$), hence itself contractible (by appealing, for instance, to the fact that the union is then a holim).

Finally one needs to show that $\holim_{\mathcal{E'}}L(\X_X(U_1),\cdots,\X_X(U_n))$ is contractible, but this is obvious since it is always $L(\X_X(U_1),\cdots,\X_X(U_n)) \sim \**$ on points of $\mathcal{E}'$.
\end{proof}

\begin{proposition}
If $L \colon\C^n \to \D$, $\D$ assumed stable, is a $(1,\cdots,1)$-reduced homotopy functor, then $L \circ \Delta$ is $n$-reduced. If, further, $L$ is symmetric, then so is $\Delta_n L$.
\end{proposition}

\begin{proof}
We need to prove that $P_{n-1}(L \circ \Delta)= \bar{T}_{n-1}^\lambda L \circ Q$ is the trivial functor. But, this is immediate from the previous result, since we can then ``interpolate'' an infinite direct colim by contractible objects (in fact, it follows that even $\bar{T}_{n-1}^{\aleph_0} L \circ Q$ is itself the zero functor, where $\aleph_0$ denotes the smallest infinite ordinal). Hence $L \circ \Delta$ is $n$-reduced.

But the case of $\Delta_n L = (L \circ \Delta)_{h \Sigma_n}$ then follows immediately, since, when the target is stable, the homotopy orbits commute with the construction of $P_{n-1}$.

\end{proof}

The converse, i.e., that $cr_n$ sends $n$-homogeneous functors to multilinear ones, is a particular case of the following:

\begin{proposition}
If $F$ is $n$-excisive then for $0 \leq m \leq n$ the functor $cr_{m+1}F$ is $(n-m)$ excisive in each variable.
\end{proposition}

\begin{proof}
Induction on $m$. The base case $m=0$ is obvious.

For the induction step, notice that, for $F_{\vee A}(X)= \hofiber(F(X\vee A) \to F(A))$ we have 
\[\bar{cr}_{m+1}F(X_1,\cdots,X_m,A) \sim \bar{cr}_{m}F_{\vee A}(X_1,\cdots,X_m)\]
whenever the $X_i, A$ are cofibrant (this is just an application of \ref{itertotfib}), and hence $cr_{m+1}F(X_1,\cdots,X_m,A) \sim cr_{m}F_{\vee A}(X_1,\cdots,X_m)$ in general.

But now the result follows since $n$-excisiveness of $F$ implies $(n-1)$-excisiveness of $F_{\vee A}$.
\end{proof}

We now turn to the task of proving that the $\Delta_n$ and $cr_n$ are indeed inverse constructions when restricted appropriately. We first tackle the case of multilinear functors.

\begin{proposition}\label{multarecross}
Let $L$ be any symmetric multilinear functor. Then $cr_n (\Delta_n L) \sim L$.
\end{proposition}

\begin{proof}

As should be expected by now, it will suffice to produce a zig zag $ \xymatrix{
L \ar@{~>}[r] & \bar{cr}_n (\Delta_n L)
}$ that is an equivalence on cofibrant objects.

First notice that
\begin{equation}
\begin{split}
\bar{cr}_n (\Delta_n L)(X_1,\cdots,X_n) & =hototfiber(\Delta_n L(\X_{X_1,\cdots,X_n}^{cr})) \\
&= hototfiber(L(\X_{X_1,\cdots,X_n}^{cr}, \cdots, \X_{X_1,\cdots,X_n}^{cr})_{h \Sigma_n}) \\
& \sim hototfiber(L(\X_{X_1,\cdots,X_n}^{cr}, \cdots, \X_{X_1,\cdots,X_n}^{cr}))_{h \Sigma_n}
\end{split}
\end{equation}
where in the last step we use that the target category $\D$ is stable.

Recalling that $\X^{cr}_{X_1,\cdots,X_d}(T) = \bigvee_{i \in S - T} X_i \vee \bigvee_{j \in T} C(X_j)$ one gets, for any map $\pi \colon \underline{n} \to S - T$ a map $L(\X_{X_1,\cdots,X_n}^{cr}, \cdots, \X_{X_1,\cdots,X_n}^{cr})(T) \to L(X_{\pi(1)},\dots,X_{\pi(n)})$. In fact, putting these all together and appealing to multilinearity of $L$, one has that the map 
\[L(\X_{X_1,\cdots,X_n}^{cr}, \cdots, \X_{X_1,\cdots,X_n}^{cr})(T) \xrightarrow{\sim} \prod_{\pi \colon \underline{n} \to S - T}{L(X_{\pi(1)},\dots,X_{\pi(n)})}\]
is a w.e. (when the $X_i$ are cofibrant).

Since these equivalences are natural in $T$, they can be reinterpreted as a map of cubes
\[L(\X^{cr},\cdots,\X^{cr}) \to \prod_{\pi \colon \underline{n} \to \underline{n}} \mathcal{Y}_{\pi}\]
where $\mathcal{Y}_{\pi}$ is the cube with $\mathcal{Y}_{\pi}(T) = L(X_{\pi(1)},\dots,X_{\pi(n)})$ if $\pi(\underline{n}) \subset T$, and $\mathcal{Y}_{\pi}(T) = \**$ otherwise.

But now one notices that when $\pi$ is not a permutation, and hence not surjective, the cube $\mathcal{Y}_{\pi}$ is cartesian (it is constant on any direction $s \notin \pi(\underline{n})$), so we can ignore those factors. On the other hand, when $\pi$ is a permutation, the cubes are $\**$ on all entries except the initial one $L(X_{\pi(1)},\dots,X_{\pi(n)})$, which is hence the tothofiber. Furthermore, these factors are freely permuted by the $\Sigma_n$ action, so one finally gets 
\[hototfiber(L(\X_{X_1,\cdots,X_n}^{cr}, \cdots, \X_{X_1,\cdots,X_n}^{cr}))_{h \Sigma_n} \sim L(X_1,\cdots,X_n)\]
as intended.
\end{proof}

\begin{remark}
Following the steps of the previous proof one obtains a more ``explicit'' description of the zig zag  $\xymatrix{
L \ar@{~>}[r]^\theta & \bar{cr}_n (\Delta_n L)}$. Indeed, consider the natural map $L(X_1,\cdots,X_n) \xrightarrow{\iota=(\iota_1,\cdots,\iota_n)} L(Z,\cdots,Z)$ where $Z=\vee X_i$ and the $\iota_i$ are the inclusions $X_i \to Z$.

The previous proof then shows that the composite $L(X_1,\cdots,X_n) \xrightarrow{\iota} L(Z,\cdots,Z)\to (L(Z,\cdots,Z))_{h \Sigma_n}$ hits precisely $\bar{cr}_n (\Delta_n L)$ (which by Corollary \ref{doublecube} is a summand of $L(Z,\cdots,Z))_{h \Sigma_n}$.
Hence we have a commutative diagram (in the homotopy category of functors\footnote{Technically, as written here, the reverse directions needed for the zigzags might only be w.e. on cofibrant objects.})
\[
\xymatrix{L(X_1,\cdots,X_n)\ar[r]^\iota \ar[d]^\theta & L(Z,\cdots,Z) \ar[d] \\
\bar{cr}_n(\Delta_n L)(X_1,\cdots,X_n) \ar[r]^\epsilon & L(Z,\cdots,Z)_{h \Sigma_n}
}\]

\end{remark}

We are now in a position to prove the converse result:

\begin{proposition}\label{homogarecross}
Let $F$ be any $n$-homogeneous functor. Then $\Delta_n (cr_n F) \sim F$.
\end{proposition}

\begin{proof}
As before, it will suffice to produce a zig zag $\xymatrix{
\Delta_n (\bar{cr}_n F) \ar@{~>}[r]^{\gamma} & F}$ that is a w.e. on cofibrant objects.

This will be the map (best viewed as a map in the homotopy category of functors)
\[\Delta_n (\bar{cr}_n F)(Z) = ((\bar{cr}_nF)(Z,\cdots,Z))_{h \Sigma_n} \xrightarrow{\gamma} F(Z)_{\Sigma_n} = F(Z)\]
induced by the map (best viewed as a map in the homotopy category of functors $\C \to \D^{\Sigma_n})$.
\[(\bar{cr}_nF)(Z,\cdots,Z) \xrightarrow{\epsilon} F(\coprod{Z}) \xrightarrow{F(f)} F(Z)\]
where $f \colon \coprod{Z} \to Z$ denotes the fold map.

To prove $\gamma$ a w.e. it suffices, by \ref{crossdetect}, to show that $\bar{cr}_n(\gamma)$ is.

In turn, from the previous result it will suffice to prove that the composite
\[\bar{cr}_n F \xrightarrow{\theta} \bar{cr}_n \Delta_n \bar{cr}_n F \xrightarrow{\bar{cr}(\gamma)} \bar{cr}_n F\]

is a w.e.. And this in turn will follow if the composite $\bar{cr}_n F(X_1,\dots,X_n) \xrightarrow{\bar{cr}_n(\gamma) \circ \theta} \bar{cr}_n F(X_1,\dots,X_n) \xrightarrow{\epsilon} F(Z)$ is homotopic to $\epsilon$ (this is because, according to \ref{doublecube}, $\epsilon$ is a homotopy injection).

This now follows from considering the following commutative diagrams
\[
\xymatrix{(\bar{cr}_n F)(X_1,\cdots,X_n) \ar[r]^{\iota} \ar[d]^{\theta} & (\bar{cr}_n F)(Z,\cdots,Z) \ar[d]\\
(\bar{cr}_n\Delta_n \bar{cr}_n F)(X_1,\cdots, X_n) \ar[r]^{\epsilon} \ar[d]^{\bar{cr}_n(\gamma)} & (\bar{cr}_n F)(Z,\cdots,Z)_{h \Sigma_n} \ar[d]^{\gamma} \\
(\bar{cr}_n F)(X_1,\cdots,X_n) \ar[r]^{\epsilon} & F(Z)
}
\]
(Notice that the right vertical composite is just the map that induced $\gamma$)
\[\xymatrix{
(\bar{cr}_n F)(X_1,\cdots,X_n) \ar[r]^{\epsilon} \ar[d]^{\iota} & F(Z) \ar[d]^{F(\coprod{\iota_i})} & \\
(\bar{cr}_n F)(Z,\cdots,Z) \ar[r]^{\epsilon} & F(\coprod{Z}) \ar[r]^{F(f)} & F(Z)}
\]
The result now follows from noticing that the composite $F(Z) \xrightarrow{\coprod{\iota_i}} F(\coprod{\iota_i}) \xrightarrow{F(f)} F(Z)$ is homotopic to the identity.

\end{proof}

\section{Delooping homogeneous functors}\label{delooping}

The goal of this section is to obtain an adequate analogue of Goodwillie's Lemma 2.2 in \cite{Goodwillie}, showing that homogeneous functors can always be delooped.

Unfortunately this will require us to make an additional hypothesis on the target category $\D$:

\begin{hypothesis}\label{hyp}
Assume that in $\D$ one has that $\P_1(S)$ shaped holims commute with {\bf countable} directed $\hocolim$s.
\end{hypothesis}

\begin{remark}

Notice that then $P_nF$ can be computed as $T_n^{\aleph_0}F$, as in \cite{Goodwillie}.
\end{remark}

\begin{lemma} \label{loop}
If $F$ is any reduced functor then, up to natural equivalence, there is a fibration sequence
\[P_n F \to P_{n-1} F \to R_n F\]
in which the functor $R_n F$ is $n$-homogeneous.
\end{lemma}

\begin{proof}
This is just a repeat of the proof of Lemma $2.2$ in \cite{Goodwillie}, which we do not repeat here as it is fairly long. One merely does the obvious changes on replacing $X \** T$ with $\X_X(T)$ and precomposing everywhere with $Q$.
\end{proof}

\begin{remark}
The reason for introducing Hypothesis \ref{hyp} is that there seems to be no clear way of generalizing Goodwillie's proof to the transfinite case, as it is rather clear what should take the place of the $\P_0(n)^i$ appearing in the proof (the obvious guess, $\P_0(n)^{\alpha}$, for $\alpha$ an ordinal, does not seem to be workable).
\end{remark}

\begin{proposition} \label{simploop}

Suppose $\D$ a pointed simplicial model category, and suppose given a diagram of pointwise fibrant functors $\C \to \D$
\[ \xymatrix{A \ar[r] \ar[d]& S \ar[d]\\
K \ar[r] & C}
\]
such that $S,K \sim \**$.

Then there exists a w.e. square of functors

\begin{equation} \label{sq}
\xymatrix{\bar{A} \ar[r] \ar[d]& Map_{\**}((I,\**),C) \ar[d]\\
\** \ar[r] & C}
\end{equation}
yielding the associated map $A' \to \Omega C$ (where $\Omega$ denotes simplicial loops).

\end{proposition}

\begin{proof}

First notice that one can freely assume the vertical maps in \ref{sq} are levelwise fibrations by performing functorial factorizations.

Next, to replace $K$ by $\**$ one simply does the pullback
\[ \xymatrix{A' \ar[r] \ar@{->>}[d] & A \ar[r] \ar@{->>}[d]& S \ar@{->>}[d] \\
\** \ar[r] & K \ar[r] & C}
\]

Next one uses mapping path objects:

\[ \xymatrix{A' \ar[d] \ar[r] &  S \ar[d] \\
R \ar@{->>}[d] \ar[r] &  Map(I,C)\times_C S \ar@{->>}[d] \\
\** \ar[r] & C}
\]

where $Z$ is defined as the pushout of the lower corner (that the map $A \to Z$ is a w.e. follows from both squares being homotopy pullback squares).

Finally, one uses the obvious w.e. map $\** \to S$ to obtain (notice $Map_{\**}((I, \**),C)= Map(I,C)\times_C \**$)

\[ \xymatrix{\bar{A} \ar[d] \ar[r] &  Map_{\**}((I,\**),C) \ar[d] \\
R \ar@{->>}[d] \ar[r] &  Map(I,C)\times_C S \ar@{->>}[d] \\
\** \ar[r] & C}
\]

\end{proof}

\begin{proposition}\label{factspectratar}

Let $\C' \subset \C$ any small subcategory, and $F\colon \C \to \D$ a homogeneous functor, where $\D$ is further assumed to be simplicial. Then $F|_{C'}$ factors (up to a ziz zag by w.e.) by $Sp(\C,\Sigma)$.

\end{proposition}

\begin{proof}
Iterative application of Lemma \ref{loop} and Proposition \ref{simploop} yields two sequences of pointwise fibrant functors $A_i$ and $A'_i$ together with w.e.'s $A_i \xrightarrow{\sim} \Omega A'_{i+1}$.

Since the projective model structure on $\D^{\C'}$ exists (Theorem 13.3.2 in \cite{Emily}), we get to that we can form $A''_{i}$ together with maps $\xymatrix{A_{i} & A''_{i} \ar@{->>}[l]_{\sim} \ar@{->>}[r]^{\sim} & A'_{i}}$.

Hence one may just as well replace $A_i$ by $A''_i$. Now inductively define $A^{(n)}_{i}$ by setting $A^{(2)}_{i}=A''_i$ and forming the pullbacks
\[\xymatrix{A^{(n+1)}_{i} \ar@{->>}[d]^{\sim} \ar[r]^{\sim} & \Omega A^{(n)}_{i+1} \ar@{->>}[d]^{\sim} \\
A^{(n)}_{i} \ar[r]^{\sim} & \Omega A^{(n-1)}_{i+1} 
}
\]
Finally one defines $A^{(\infty)}_{i}=\lim_{n} A^{(n)}_i$, which come equipped with the desired spectrum maps $A^{(\infty)}_{i} \xrightarrow{\sim} \Omega A^{(\infty)}_{i+1}$.

\end{proof}

\section{Factoring multilinear symmetric functors through spectra}\label{factspectra}

In this section we use the notion of spectra described in \ref{assumpspectra}, along with the assumptions described there. Namely, we always assume that the spectra categories $Sp(\C)$ are cofibrantly generated, that they have a localization functor as in Proposition \ref{localization}, and that stable w.e.s are detected by the $\tilde{\Omega}^{\infty-n}$ functors.

\begin{definition}
Consider a (multi-)functor $F \colon \C^n \to \D$. There is an obvious left action\footnote{Which by abuse of notation we just denote by the corresponding elements $\sigma \in \Sigma_n$} of $\Sigma_n$ on $\C^n$, and we call the functor $F$ {\bf symmetric} if there are natural transformations
\[F \xrightarrow{\mu_\sigma} F\circ \sigma\]
satisfying $\mu_{\sigma \sigma'}=\mu_{\sigma}\sigma' \circ \mu_\sigma'$.

Furthermore, given symmetric functors $F,G$, we define a {\bf natural transformation} between them to be a natural transformation $\eta \colon F \to G$ compatible with $\mu_{\sigma,F}, \mu_{\sigma,G}$ in the obvious way.
\end{definition}

Our goal in this section is to prove a theorem of the following kind.

\begin{theorem}\label{multfact} 
Suppose given a symmetric functor $F \colon \C^{\times n} \to \D$ between nice enough model categories which is linear in each coordinate. Then $F$ is of the form\footnote{That is to say, there is a zig zag of w.e.s between the two.} $\tilde{\Omega}^\infty \circ \bar{F} \circ (\tilde{\Sigma}^\infty)^{\times n}$, where $\bar{F} \colon Sp(\C)^{\times n} \to Sp(\D)$ is itself a symmetric multilinear functor. 
\end{theorem}

We will use the following generalization of the previous definition.

\begin{definition}
Let $\C, \D$ be categories acted on the left by $\Sigma_n$ (or more generally any group $G$). Then a functor $F\colon \C \to \D$ is called {\bf symmetric} if there are natural transformations
\[\sigma \circ F \xrightarrow{\mu_\sigma} F\circ \sigma\]
satisfying $\mu_{\sigma \sigma'}=\mu_{\sigma}\sigma' \circ \sigma \mu_\sigma'$.

Furthermore, given symmetric functors $F,G$, we define a {\bf natural transformation} between them to be a natural transformation $\eta \colon F \to G$ compatible with $\mu_{\sigma,F}, \mu_{\sigma,G}$ in the obvious way.

\end{definition}

\begin{remark}
Notice that given $\Sigma_n$-categories $\C,\D,\E$  and symmetric functors $\C \xrightarrow{F} \D \xrightarrow{G} \E$ the composite $G\circ F$ is also a symmetric functor with $\mu_{\sigma,G\circ F}$ the composite $\sigma \circ G \circ F \xrightarrow{\mu_{\sigma,G}F}  G  \circ \sigma \circ F \xrightarrow{G \mu_{\sigma,F}} G \circ F \circ \sigma $
\end{remark}

\begin{proposition}
Suppose $\C$ is a cofibrantly generated model category acted on the left by a group $G$. 

Then the functorial factorizations can be chosen compatible with the $G$ action.osen compatible with the $G$ action.
\end{proposition}

\begin{proof}
Letting $I$, $J$ be the generating cofibrations and trivial cofibrations for $\C$, one extends these. Namely, $I^G=\{(g,g(f))\colon g \in G, f\in I\}$, and analogously for $J^G$. Performing the Quillen small object arguments for these sets will then yield the desired invariant factorizations.

\end{proof}

\begin{remark}
We hence assume that the fibrant/cofibrant replacements performed in such a category are $G$ invariant.
\end{remark}

We now construct a first avatar of $\bar{F}$, which we call $\bar{F}'$.

Notice that $Sp(\D,\Sigma^n)$ has an obvious action by the group $\Sigma_n$ induced by its action on $(S^1)^{\wedge n}$. We now define:
\begin{align}
\bar{F}'\colon &Sp(\C,\Sigma)^n\to  Sp(\D,\Sigma^n) \\
& ((X^1_i)_{i\in \mathbb{N}}, \dotsc,(X^n_i)_{i\in \mathbb{N}})\mapsto (F(X^1_i,\dotsc,X^n_i))_{i \in \mathbb{N}}
\end{align}
with structure maps for $\bar{F}'((X^1_i)_{i\in \mathbb{N}}, \dotsc,(X^n_i)_{i\in \mathbb{N}})$ being given by the composite 
\[(S^1)^{\wedge n}\wedge F(X^1_i,\dotsc,X^n_i) \to F(S^1 \wedge X^1_i,\dotsc,S^1 \wedge X^n_i) \to F(X^1_{i+1},\dotsc,X^n_{i+1})\]
where the first map comes from the multilinear structure of $F$, and uses the sphere coordinates in the obvious {\bf fixed} order.

\begin{proposition}
If $F$ is a symmetric functor $\C^n\to \D$, then so is $\bar{F}'\colon Sp(\C,\Sigma)^n\to  Sp(\D,\Sigma^n)$ (now taking into account that both the source and target are categories with a $\Sigma_n$ action).
\end{proposition}

\begin{proof}
We require maps $\sigma \circ \bar{F}' \xrightarrow{\mu_\sigma} \bar{F}' \circ \sigma$. Since the $\Sigma_n$ action on $Sp(\D,\Sigma^n)$ does not change the underlying sequence of the spectrum, but merely twists the structure maps, such maps are immediately obtained levelwise by applying the $\nu_{\sigma,F}$. That this is compatible with the structure maps of the spectra is precisely guaranteed by the referred twisting.
\end{proof}

\begin{proposition}
$\tilde{\Omega}^{\infty}\colon Sp(\D,\Sigma^n) \to \D$ is a symmetric functor.
\end{proposition}

\begin{proof}
We need only show that $\Omega^{\infty}$ itself (i.e. the non derived version) is symmetric, since fibrant replacements are chosen symmetric.

We now need to produce maps $\Omega^\infty \xrightarrow{\mu_\sigma} \Omega^\infty \circ \sigma$. But both sides, on a spectrum $(X_i)_{i \in \mathbb{N}}$, are given as colimits $\varinjlim Map((S^1)^{\wedge i},X_i)$, the only difference being that the structure maps are twisted by the $\Sigma_n$ action. That action then gives a map between the direct systems, and hence the required natural transformation.
\end{proof}

\begin{remark}
The above statement would be true replacing $(S^1)^{\wedge n}$ by any other $\Sigma_n$ simplicial pointed set.
\end{remark}

We now prove it is $F$ equivalent to $\Omega^\infty \circ \bar{F}' \circ (\Sigma^\infty)^{\times n}$.

\begin{proposition}
Let $F$ be a multi-simplicial and multilinear homotopy functor.

There is a natural zig zag of equivalences of symmetric functors
\[F \rightsquigarrow \tilde{\Omega}^\infty \circ \bar{F}' \circ (\tilde{\Sigma}^\infty)^{\times n}\]

\end{proposition}

\begin{proof}
Standard tricks allow us to just reduce to comparing $F$ and $\Omega^\infty \circ \bar{F}' \circ (\Sigma^\infty)^{\times n}$ (namely, one can replace $\bar{F}'$ so as to take fibrant values in the projective level model structure and replace $F$ to match its $0$-th level functor) over cofibrant values.

$\Omega^\infty \circ \bar{F}' \circ (\Sigma^\infty)^{\times n}$ is given by a colimit of functors $Map(S^{ni},F(S^i \wedge X^1,\dotsc,S^i \wedge X^n))$, of which $F$ is the $0$-th functor. This yields the natural transformation, which is easily seen to be symmetric. That it is also a w.e. is just an immediate consequence of the multilinearity of $F$. 
\end{proof}

\begin{remark}
A certain amount of care must be taken when reading the previous result. This is because it is not immediately clear that the functor $F'$ is actually homotopical with respect to \emph{stable} equivalences on the spectra categories. Indeed, a priori, the only type of equivalences that are obviously preserved by $\bar{F}'$ are levelwise equivalences of spectra.mount of care must be taken when reading the previous result. This is because it is not immediately clear that the functor $F'$ is actually homotopical with respect to \emph{stable} equivalences on the spectra categories. Indeed, a priori, the only type of equivalences that are obviously preserved by $\bar{F}'$ are levelwise equivalences of spectra.

The next result provides conditions in which this holds.

Before stating we point out, however, that in the previous result $\tilde{\Omega}^{\infty}$ can be defined using either a stable fibrant replacement or a mere levelwise fibrant replacement (for some fixed levelwise model structure). The reason is that, as the proof of the result shows, the image of $\bar{F}' \circ \tilde{\Sigma}^{\infty}$ is already an $\Omega$-spectrum up to level equivalence (or, in other words, the spaces of the spectrum already have the right homotopy type).

In other words, no problems are caused by viewing $Sp(\D)$ as equipped with the stable model structure.
\end{remark}

\begin{proposition}
Let $F$ be a multilinear multihomotopical functor $\C^{n}\to \D$ as before, and suppose further that $F$ preserves filtered homotopy colimits (in each variable).
Then $\bar{F}'$ sends tuples $(f_1,\cdots,f_n)$, where the $f_i$ are either stable w.e.s between cofibrant objects or identities, to stable w.e.s., and hence $\bar{F}' \circ Q^n$  is a homotopical functor with respect to the stable model structures.

Furthermore, $\bar{F}' \circ Q^n$ is a multilinear functor.
\end{proposition}

\begin{proof}
First notice that any diagrams in a spectra category that are hocolim/holim diagrams with respect to levelwise w.e.s are also hocolim/holim diagrams with respect to stable w.e.s.

Hence $\bar{F}'$ preserves level direct hocolims, and sends level homotopy pushout diagrams to level homotopy pullback diagrams. But since its target is stable (we show later in the section that $Sp(\D,\Sigma^n)$ is appropriately equivalent to $Sp(\D,\Sigma$)) the latter are stable homotopy pushout diagrams.

Now by the 2 out of 3 property one need only show that $F$ sends the localization maps $X\to X_{loc}$ whenever $X$ is cofibrant (this is because w.e.s between local objects are always levelwise).

But by Proposition \ref{localization} those localization maps are always transfinite composition of pushouts of the maps in $\widetilde{\Delta S}$. As these transfinite compositions and pushouts are in fact homotopy transfinite compositions and homotopy pushouts (in the level structure on $Sp(\C,\Sigma^1)^{\times n}$), and $F$ preserves these (when regarding $Sp(\D,\Sigma^n)$ as having the stable model structure), it suffices to show that $F$ sends maps that are coordinatewise in $\widetilde{\Delta S}$ or identities to w.e.s. But this is now obvious from the description of $\widetilde{\Delta S}$, because those are always maps that are w.e.s in high enough degrees.

For the second part, one needs to show that pushout diagrams are sent to pullback diagrams (or equivalently pushout diagrams). But since any stable pushout diagram is stably equivalent to a level pushout diagram this is now obvious from the previous discussion.

\end{proof}

We have now essentially proven for $\bar{F}'$ the desired properties for $\bar{F}$. The main difference between the two is their target categories, $Sp(\D,\Sigma^n)$ and $Sp(\D,\Sigma)$. But we claim that the two are Quillen equivalent (through a zig zag), as $\Sigma_n$ model categories (the latter having a trivial $\Sigma_n$ action), and in a way compatible with the $\Omega^\infty$ functors, and from this the result follows.

To see why this is plausible, notice that, topologically, $S^n$ decomposes, as a $\Sigma_n$-object, into $S^1 \wedge S^{\bar{\rho}}$, with the action on $S^1$ being trivial and $S^{\bar{\rho}}$ being the reduced regular representation. But then, since inverting $S^1\wedge \bullet$ also inverts $S^{\bar{\rho}} \wedge \bullet$ (notice that the presence of a $\Sigma_n$ action determines the action on the category of spectra, but is irrelevant for the model category structure itself), plausibility follows.

We now set to prove this. As a first step we need to replace the simplicial $S^n$, which lacks any decomposition as above, with something more amenable. That this can be done follows from the following (particular case of a) theorem of Hovey:

\begin{theorem} \label{spquillenequiv}
Let $\C$ be a simplicial model category for which spectra can be defined, as described in \ref{assumpspectra}.

Suppose $f\colon A \to B$ a w.e. of pointed simplicial sets. Then, provided that the domains of the generating cofibrations of $\C$ are cofibrant or that the induced natural transformation of functors $X \wedge \bullet \to Y\wedge \bullet$ is a pointwise w.e., $f$ induces a Quillen equivalence
\[Sp(\C,A)\rightleftarrows Sp(\C,B)\]
\end{theorem}

\begin{proof}
This is a particular case of Theorem 5.5 in \cite{Hovey}. It is worth noting that in the proof Hovey uses Proposition 2.3 from that paper, which supposes the involved categories to be left proper cellular. However, careful examination of the proofs shows that left properness and cellularity are only used to conclude the existence of the left Bousfield localization model structure, and hence the result still holds provided one knows those to exist.
\end{proof}

\begin{remark}
The Quillen equivalence above is also compatible with the $\Omega^\infty$s. This is clear for the right adjoint (which is a forgetful functor) from the obvious map $\varinjlim Map(B^{\wedge i},X_i)\to \varinjlim Map(A^{\wedge i},X_i)$ being a w.e. when $(X_i)$ is a projective fibrant spectrum. The result also follows for the left adjoint since (it's derived functor) is the inverse of the (derived functor) of the right adjoint.
\end{remark}

Using the above result we are then allowed to replace $S^n$ by a $\Sigma_n$ w.e. simplicial set of the form $S^1\wedge S^{\bar{\rho}}$, possibly resorting to a ziz zag of $\Sigma_n$ w.e. $S^n\gets C \to F \gets S^1\wedge S^{\bar{\rho}}$.

We take this op

We will now be finished by proving the following:

\begin{proposition}

Suppose $\D$ is in the finitely generated case described in \cite{Hovey}, so that stable w.e. in the spectra categories are determined at the level of the $\Omega^{\infty-n}$ functors.

Then there is a $\Sigma_n$ Quillen equivalence 

\[Sp(\D,S^1)\rightleftarrows Sp(\D,S^1\wedge S^{\bar{\rho}})\]

with left adjoint given by 
\[L((X_i)_{i \in \mathbb{N}})=((S^{\bar{\rho}})^{\wedge i}\wedge X_i)_{i \in \mathbb{N}}\]
and right adjoint given by
\[R((Y_i)_{i \in \mathbb{N}})=(Map((S^{\bar{\rho}})^{\wedge i}, Y_i))_{i \in \mathbb{N}}\]

\end{proposition}

\begin{proof}
First we show this is a Quillen adjuntion. For the projective model structures we clearly have a Quillen adjunction, since the left adjoint carries generating cofibrations/triv. cofibrations to cofibrations/triv. cofibrations (namely, a generating cofib/triv cofib of the form $F_n(f)$ is sent to $F_n(f \wedge (S^{\bar{\rho}})^{\wedge n})$).

Similarly, the localizing set of maps $F_{n+1}(X\wedge S^1) \to F_n{X}$ is sent to $F_{n+1}(X\wedge S^1\wedge (S^{\bar{\rho}})^{\wedge n+1}) \to F_n({X}\wedge (S^{\bar{\rho}})^{\wedge n})$, also a stable equivalence. Hence one has indeed a Quillen adjunction.

We now turn to proving that this is indeed a Quillen equivalence.

As usual, consider a cofibrant $S^1$ spectrum $X$, a fibrant $S^1\wedge S^{\bar{\rho}}$ spectrum $Y$, so that we want to show a map $LX \to Y$ is w.e. iff the map $X \to RY$ is.

We now use our hypothesis that w.e. are determined at the $\tilde{\Omega}^{\infty -n}$ level. For $LX \to Y$ this yields maps (where $\tilde{\Omega}$ denotes true homotopical loops)
\[\hocolim_i \tilde{\Omega}^{i(1+\bar{\rho})}((S^{\bar{\rho}})^{i+n}\wedge X_{i+n}) \to \hocolim_i \tilde{\Omega}^{i(1+\bar{\rho})}Y_{i+n}\]
while for $X \to RY$ it yields maps
\[\hocolim_i \tilde{\Omega}^nX_{i+n} \to \hocolim_i \tilde{\Omega}^i Map((S^{\bar{\rho}})^{\wedge i}, Y_{i+n})= \hocolim_i \tilde{\Omega}^{i(1+\bar{\rho})}Y_{i+n}.\]

But since this last map factors through the obvious map \[\hocolim_i \Omega^nX_{i+n} \to \hocolim_i \Omega^{i(1+\bar{\rho})}((S^{\bar{\rho}})^{i+n}\wedge X_{i+n}),\]
it remains to show that these maps are w.e. and, without loss of generality, to do so for $i=0$.

Next one notices that the $\Sigma_n$ actions are irrelevant as far as detecting w.e.s, so that we will be done if we prove the analogous result for the analogous adjuntction of model categories\footnote{To see this one may find a bifibrant $S^{\bar{\rho}} \xrightarrow{\sim} C \xleftarrow{\sim} (S^1)^{n-1}$ and apply \ref{spquillenequiv}.} given by the obvious analogous formulae:

\[Sp(\D,S^1)\rightleftarrows Sp(\D,S^1\wedge (S^1)^{n-1})\]

In order words, it now remains to prove the map
\[\hocolim_i \tilde{\Omega}^iX_{i} \to \hocolim_i \tilde{\Omega}^{i n}((S^{n-1})^{i} \wedge X_{i}) = \hocolim \tilde{\Omega}^{i}\tilde{\Omega}^{(n-1)i}\Sigma^{(n-1)i} X_{i}\]
is a w.e..
At this point a certain amount of care must be taken with the suspension/loop coordinates. Henceforth assume suspension coordinates ordered left to right (and loop coordinates ordered inversely to make the notation compatible with applying adjunction unit maps).

The maps defining the left $\hocolim$ are of the form
\begin{equation}
\begin{split}
\tilde{\Omega}^{i}\tilde{\Omega}^{(n-1)i}\Sigma^{(n-1)i} X_{i} & \to \tilde{\Omega}^{i}\tilde{\Omega}^{(n-1)i}\tilde{\Omega}^n \Sigma ^n \Sigma^{(n-1)i} X_{i} \\ & 
\to \tilde{\Omega}^{i} \tilde{\Omega} \tilde{\Omega}^{(n-1)i}\tilde{\Omega}^{\{n,\cdots,2\}} \Sigma^{\{2,\cdots,n\}} \Sigma^{(n-1)i} \Sigma X_{i} \\
& \to \tilde{\Omega}^{i+1} \tilde{\Omega}^{(n-1)(i+1)} \Sigma^{(n-1)(i+1)} X_{i+1}
\end{split}
\end{equation}
where the first map is the unit map, and the second map moves the first coordinate in the inner $\Sigma^n$ until it is adjacent to $X_i$ (and, by symmetry, moves the associated loop coordinate in the opposite direction), and the final map is induced by $\Sigma X_i \to X_{i+1}$ (we also perform some rewriting of terms, \emph{but no further shuffling}).

By contrast, consider the hocolimit with the same terms, but with the reshufflings in the middle step removed, so that the outer suspension coordinate is the one used in the final map.

These two hocolimits are of course equivalent, since when they are computed using only the intermediate terms $\tilde{\Omega}^{i}\tilde{\Omega}^{(n-1)i}\tilde{\Omega}^n \Sigma ^n \Sigma^{(n-1)i} X_{i}$, the direct (sub)systems are actually isomorphic. 

Now notice that the unshuffled direct system can be identified with $\tilde{\Omega}^\infty$ for a new $S^n$-spectrum $(S^{(n-1)i} \wedge X_i)_{i \in \mathbb{N}}$ with structure maps $S^n \wedge S^{(n-1)i} \wedge X_i = S^{(n-1)(i+1)} \wedge S^1 \wedge X_i \to S^{(n-1)(i+1)} \wedge X_{i+1}$ (with no shuffling used).

It then suffices to show the map relating this last direct system with that for $\Omega^{\infty}$ of $(X_i)_{i \in \mathbb{N}}$ induces a w.e. on hocolims.

To see this consider the diagram
\[
\xymatrix{X_0 & S^1 \wedge X_0 & S^2 \wedge X_0 & S^3 \wedge X_0 & S^4 \wedge X_0  & \dots \\
& X_1 & S^1 \wedge X_1 & S^2 \wedge X_1 & S^3 \wedge X_1 & \dots \\
& & X_2 & S^1 \wedge X_2 & S^2 \wedge X_2 & \dots \\
& & & \ddots & & },
\]
where the lines represent the spaces of $S^1$-spectra, and the columns obvious maps between these spectra. Now consider applying $\tilde{\Omega}^n$ to the $n$-th column, so that one has a full diagram. Notice that the direct system for $\tilde{\Omega}^\infty$ of $(X_i)_{i \in \mathbb{N}}$ is the line of slope 1 in this diagram, while the direct system for $\tilde{\Omega}^{\infty}$ of $(S^{(n-1)i} \wedge X_i)_{i \in \mathbb{N}}$ is the line of slope $n$. But since tracking definitions shows the map induced between those direct systems is given by traveling horizontally in the diagram above, the result finally follows by cofinality of those lines in the diagram.


\end{proof}

\end{document}